\documentclass[a4paper]{amsart} 
\usepackage{amssymb} 
\usepackage{stmaryrd}
\usepackage[mathcal]{euscript}
\usepackage{tikz}
\usepackage{tikz-cd}
\usepackage{hyperref}
\usepackage[all]{xy}
\SelectTips{cm}{}

\makeatletter
\let\@wraptoccontribs\wraptoccontribs
\makeatother

\title{Completing perfect complexes}

\author[Henning Krause]{Henning Krause}

\thanks{Tobias Barthel was supported by the the Danish National
  Research Foundation (DNRF92) and the European Union's Horizon 2020
  research and innovation programme under the Marie Sklodowska-Curie
  grant agreement No.~751794.}

\contrib[with appendices by]{Tobias Barthel and Bernhard Keller}

\address{T.~B.: Max Planck Institut f\"ur Mathematik, Vivatsgasse 7, 53111 Bonn, Germany}
\email{tbarthel@mpim-bonn.mpg.de}

\address{B.~K.: Universit\'e Paris Diderot -- Paris~7, UFR de
  Math\'ematiques, Institut de Math\'ematiques de Jussieu--PRG, UMR
  7586 du CNRS, Case 7012, B\^atiment Sophie Germain, 75205 Paris
  Cedex 13, France}
\email{bernhard.keller@imj-prg.fr}

\address{H.~K.: Fakult\"at f\"ur Mathematik,
Universit\"at Bielefeld, 33501 Bielefeld, Germany}
\email{hkrause@math.uni-bielefeld.de}

\theoremstyle{plain}
\newtheorem{thm}{Theorem}[section]
\newtheorem{prop}[thm]{Proposition}
\newtheorem{lem}[thm]{Lemma} 

\newtheorem{cor}[thm]{Corollary}

\theoremstyle{definition}
\newtheorem{defn}[thm]{Definition}
\newtheorem{exm}[thm]{Example}

\theoremstyle{remark}
\newtheorem{rem}[thm]{Remark}

\numberwithin{equation}{section}

\newcommand{\Add}{\operatorname{Add}}

\newcommand{\art}{\operatorname{art}}

\newcommand{\Cau}{\operatorname{Cauch}}

\newcommand{\coh}{\operatorname{coh}}

\newcommand{\colim}{\operatorname*{colim}}

\newcommand{\End}{\operatorname{End}}

\newcommand{\Ext}{\operatorname{Ext}}

\newcommand{\fl}{\operatorname{fl}}

\newcommand{\Fun}{\operatorname{Fun}}

\newcommand{\hocolim}{\operatorname*{hocolim}}
\newcommand{\Hom}{\operatorname{Hom}}

\newcommand{\id}{\operatorname{id}}

\renewcommand{\Im}{\operatorname{Im}}

\newcommand{\Ind}{\operatorname{Ind}}
\newcommand{\inj}{\operatorname{inj}}
\newcommand{\Inj}{\operatorname{Inj}}

\newcommand{\Ker}{\operatorname{Ker}}

\newcommand{\Lex}{\operatorname{Lex}}

\renewcommand{\min}{\operatorname{min}}
\renewcommand{\mod}{\operatorname{mod}}
\newcommand{\Mod}{\operatorname{Mod}}

\newcommand{\Ph}{\operatorname{Ph}}
\newcommand{\proj}{\operatorname{proj}} 

\newcommand{\Qcoh}{\operatorname{Qcoh}}

\newcommand{\soc}{\operatorname{soc}}

\newcommand{\St}{\operatorname{St}}

\newcommand{\Thick}{\operatorname{Thick}}

\newcommand{\Ab}{\mathrm{Ab}}

\newcommand{\op}{\mathrm{op}}
\newcommand{\per}{\mathrm{per}}

\newcommand{\Set}{\mathrm{Set}}

\newcommand{\two}{\mathbb I}

\newcommand{\iso}{\xrightarrow{_\sim}}

\newcommand{\Lotimes}{\otimes^{\mathbf L}}

\newcommand{\lto}{\longrightarrow}
\newcommand{\smatrix}[1]{\left[\begin{smallmatrix}#1\end{smallmatrix}\right]}

\newcommand{\xto}{\xrightarrow}

\DeclareRobustCommand\longtwoheadrightarrow
     {\relbar\joinrel\twoheadrightarrow}

\def\A{\mathcal A} 
 
\def\C{\mathcal C}
\def\D{\mathcal D} 
\def\E{\mathcal E} 
\def\F{\mathcal F} 
\def\G{\mathcal G}
\def\I{\mathcal I}

\def\M{\mathcal M}
\def\P{\mathcal P}
\def\calS{\mathcal S} 
\def\T{\mathcal T}

\def\X{\mathcal X}

\def\bfD{\mathbf D} 
\def\bfK{\mathbf K}

\def\bbN{\mathbb N}

\def\bbQ{\mathbb Q} 
\def\bbR{\mathbb R} 
\def\bbU{\mathbb U} 
\def\bbX{\mathbb X} 
 
\def\bbZ{\mathbb Z}

\def\a{\alpha}
\def\b{\beta}

\def\p{\phi}

\def\s{\sigma}
\def\t{\tau}

\def\Ga{\Gamma}
\def\La{\Lambda}
\def\Si{\Sigma}

\DeclareMathOperator{\fib}{fib}
\newcommand{\sgen}[4]{\overline{\langle {#1}\rangle}_{#2}^{#3,#4}}
\def\cO{\mathcal O}
\newcommand{\bfR}{\mathbf R}
\usepackage{todonotes}

\begin{document}

\begin{abstract}
  This note proposes a new method to complete a triangulated category,
  which is based on the notion of a Cauchy sequence. We apply this to
  categories of perfect complexes. It is shown that the bounded
  derived category of finitely presented modules over a right coherent
  ring is the completion of the category of perfect complexes. The
  result extends to non-affine noetherian schemes and gives rise to a
  direct construction of the singularity category. The parallel theory
  of completion for abelian categories is compatible with the
  completion of derived categories.

  There are three appendices. The first one by Tobias Barthel
  discusses the completion of perfect complexes for ring spectra. The
  second one by Tobias Barthel and Henning Krause refines for a
  separated noetherian scheme the description of the bounded derived
  category of coherent sheaves as a completion.  The final appendix by
  Bernhard Keller introduces the concept of a morphic enhancement for
  triangulated categories and provides a foundation for
  completing a triangulated category.
\end{abstract}



\dedicatory{Dedicated to the memory of Ragnar-Olaf Buchweitz.}

\date{October 30, 2019}
\maketitle

\section{Introduction}

This note proposes a new method to complete a triangulated category,
and we apply this to categories of perfect complexes
\cite{Il1966}. For any category $\C$, we introduce its \emph{sequential
  completion} $\widehat\C$, which is a categorical analogue of the
construction of the real numbers from the rationals via equivalence
classes of Cauchy sequences, following Cantor and M\'eray
\cite{Ca1872,Me1869}.

When a ring $\La$ is right coherent, then the category $\mod\La$ of
finitely presented modules is abelian and one can consider its bounded
derived category $\bfD^b(\mod\La)$, which contains the category of
perfect complexes $\bfD^\per(\La)$ as a full triangulated
subcategory. The following theorem describes $\bfD^b(\mod\La)$ as a
completion of $\bfD^\per(\La)$.

\begin{thm}\label{th:main}
  For a right coherent ring $\La$ there is a canonical triangle
  equivalence
  \begin{equation*}
    \widehat{\bfD^\per(\La)}^b\xto{\ _\sim\ }\bfD^b(\mod\La)
  \end{equation*}
which sends a Cauchy sequence in $\bfD^\per(\La)$ to its colimit.
\end{thm}

The description of $\bfD^b(\mod\La)$ as a completion extends to
non-affine schemes. Thus for a noetherian scheme $\bbX$ there is a
canonical triangle equivalence
\[\widehat{\bfD^\per(\bbX)}^b\xto{\ _\sim\ }\bfD^b(\coh\bbX).\]
In particular, this provides a direct construction of the singularity
category (in the sense of Buchweitz and Orlov \cite{Bu1987,Or2004}) as
the Verdier quotient
\[\frac{\widehat{\bfD^\per(\bbX)}^b}{\bfD^\per(\bbX)}.\]

The completion $\widehat\C$ of a category $\C$ comes
with an embedding $\C\to \widehat\C$ so that the objects in
$\widehat\C$ are precisely the colimits of Cauchy sequences in $\C$,
and $\widehat\C$ identifies with a full subcategory of the
ind-completion of $\C$ in the sense of Grothendieck and Verdier
\cite{GV}.

When $\C$ is triangulated, there is a natural finiteness condition
such that $\widehat\C$ inherits a triangulated structure with exact
triangles given as colimits of Cauchy sequences of exact triangles in
$\C$. This involves the notion of a phantom morphism and Milnor's
exact sequence \cite{Mi1962}. In order to explain this, let us assume
for simplicity that $\C$ identifies with the full subcategory of
compact objects of a compactly generated triangulated category $\T$.
Fix a class $\X$ of sequences $X_0\to X_1\to X_2\to\cdots$ in $\C$
that is stable under suspensions. We consider their homotopy colimits
and have
\[\Ph(\hocolim_iX_i, \hocolim_jY_j)=0\quad\text{for all}\quad X,Y\in\X\]
if and only if
\[ \sideset{}{^{1}}\lim_i\colim_j\Hom(X_i,Y_j)=0\quad\text{for
    all}\quad X,Y\in\X,\] where $\Ph(U,V)$ denotes the set of phantom
morphisms $U\to V$. It is this finiteness condition which is satisfied
for categories of perfect complexes, and it enables us to establish a
triangulated structure for the completion of $\C$ with respect to
$\X$.

The idea of completing a triangulated category $\C$ is not new; the
method is always to identify the completion $\D$ with a category of
certain cohomological functors $\C^\op\to \Ab$. Note that the category
of all cohomological functors is equivalent to the ind-completion of
$\C$.  In most cases, $\C$ identifies with the category of compact
objects of a compactly generated triangulated category $\T$, and $\D$
is another triangulated subcategory of $\T$. Let us mention the paper
of Neeman \cite{Ne1992a} that addresses the question when a category
of cohomological functors carries the structure of a triangulated
category. In \cite{Ro2008}, Rouquier identifies various natural
choices of cohomological functors $\C^\op\to \Ab$. The recent work of
Neeman \cite{Ne2018,Ne2018b} employs the notion of `approximability';
it is crucial for understanding the case of non-affine schemes and
recommended as an alternative approach via Cauchy sequences.

We also include a discussion of completions for abelian
categories. Again, some finiteness condition is needed so that the
completion is abelian. For instance, we show for a noetherian algebra
$\La$ over a complete local ring that the completion of the category
$\fl\La$ of finite length modules identifies with the category of
artinian $\La$-modules. Using Matlis duality, this yields for
$\Ga=\La^\op$ triangle equivalences
\[\bfD^b(\mod\La)^\op\xto{\ _\sim\ } \bfD^b(\widehat{\fl\Ga})\xto{\
    _\sim\ }\widehat{\bfD^b(\fl\Ga)}^b.\]

This paper has three appendices. The first one by Tobias Barthel
discusses completions for stable homotopy categories. In particular,
we see that Theorem~\ref{th:main} generalises to ring spectra.

The second one by Tobias Barthel and Henning Krause refines for a
separated noetherian scheme the description of the bounded derived category of
coherent sheaves as a completion. It is shown that the objects are
precisely the colimits of Cauchy sequences of perfect complexes that
satisfy an intrinsic boundedness condition.

In the final appendix, Bernhard Keller introduces the notion of a
\emph{morphic enhancement} of a triangulated category, following
\cite{Ke1991}. This allows us to capture the notion of standard
triangle and of coherent morphism between standard triangles,
generalising analogous approaches via stable model categories or
stable derivators. Morphic enhancements provide a setting for turning
a completion into a triangulated category. In fact, we see that in
Theorem~\ref{th:main} the completion of the morphic enhancement of
$\bfD^\per(\La)$ identifies with the morphic enhancement of
$\bfD^b(\mod\La)$.

After completion of this work, Neeman published a survey
\cite{Ne2019} which discusses metrics in triangulated categories,
following work of Lawvere from the 1970s. Completing with respect to
such metrics yields an alternative method of completing triangulated
categories; it does not depend on an enhancement and we recommend a
comparison.

\subsection*{Acknowledgement}

This work benefited from discussions at an Oberwolfach workshop in
March 2018. Following the spirit of this workshop, it is intended as a
contribution of potential common interest to stable homotopy theory,
representation theory, and algebraic geometry.

I wish to thank Amnon Neeman for various helpful comments on this
work, in particular for drawing my attention to the related notion of
`approximability', for providing the proof of
Lemma~\ref{le:scheme-bounded}, and for pointing out problems in some
previous versions of this manuscript. Also, the interest and comments
of Greg Stevenson are very much appreciated.

I am grateful to Tobias Barthel and Bernhard Keller for many
stimulating comments and for agreeing to include their ideas in form
of an appendix.

Finally, I wish to thank an anonymous referee for various helpul comments
concerning the exposition.

\section{The sequential completion of a category}

Let $\bbN=\{0,1,2,\ldots\}$ denote the set of natural numbers, viewed
as a category with a single morphism $i\to j$ if $i\le j$.

Now fix a category $\C$ and consider the category $\Fun(\bbN,\C)$ of
functors $\bbN\to\C$.  An object $X$ is nothing but a sequence of
morphisms $X_0\to X_1\to X_2\to \cdots$ in $\C$, and the morphisms
between functors are by definition the natural transformations.  We
call $X$ a \emph{Cauchy sequence} if for all $C\in\C$ the induced map
$\Hom(C,X_i)\to\Hom(C,X_{i+1})$ is invertible for $i\gg 0$. This
means:
\[\forall\; C\in\C\;\; \exists\; n_C\in\bbN\;\; \forall\; j\ge i\ge
  n_C\;\; \Hom(C,X_i)\xto{_\sim}\Hom(C,X_j).\]

Let $\Cau(\bbN,\C)$ denote the full subcategory consisting of all
Cauchy sequences.  A morphism $X\to Y$ is \emph{eventually invertible}
if for all $C\in\C$ the induced map $\Hom(C,X_i)\to\Hom(C,Y_i)$ is
invertible for $i\gg 0$. This means:
\[\forall\; C\in\C\;\; \exists\; n_C\in\bbN\;\; \forall\;  i\ge
  n_C\;\; \Hom(C,X_i)\xto{_\sim}\Hom(C,Y_i).\] Let $S$ denote the class
of eventually invertible morphisms in $\Cau(\bbN,\C)$.

\begin{defn}
  The \emph{sequential completion} of $\C$ is the category
  \[\widehat\C:= \Cau(\bbN,\C)[S^{-1}]\] that is obtained from the
  Cauchy sequences by formally inverting all eventually invertible
  morphisms, together with the \emph{canonical functor}
  $\C\to\widehat\C$ that sends an object $X$ in $\C$ to the constant
  sequence $X\xto{\id} X\xto{\id} \cdots$.
\end{defn}

A  sequence $X\colon\bbN\to\C$ induces a functor
\[\widetilde X\colon\C^\op\lto\Set,\quad C\mapsto\colim_i\Hom(C,X_i),\]
and this yields a functor
\[\widehat\C\lto\Fun(\C^\op,\Set),\quad X\mapsto\widetilde X,\]
because the assignment $X\mapsto\widetilde X$ maps eventually invertible
morphisms to isomorphisms.  We will show that this functor is fully
faithful.

Let $\D$ be a category and $S$ a class of morphisms in $\D$.  There is
an explicit description of the localisation $\D[S^{-1}]$ provided that
the class $S$ \emph{admits a calculus of left fractions} in the sense
of \cite{GZ1967}, that is, the following conditions are satisfied:
\begin{enumerate}
\item[(LF1)] The identity morphism of each object is in $S$. The
  composition of two morphisms in $S$ is again in $S$.
\item[(LF2)] Each pair of morphisms $X'\xleftarrow{\s}X\to Y$ with
  $\s\in S$ can be completed to a commutative diagram 
  \[
\begin{tikzcd}
X\arrow{r}\arrow[swap]{d}{\s}&
Y\arrow{d}{\t}\\
X'\arrow{r}&Y'
\end{tikzcd}
\]
such that $\t\in S$.
\item[(LF3)] Let $\a,\b\colon X\to Y$ be morphisms. If there is
  $\s\colon X'\to X$ in $S$ such that $\a\s=\b\s$, then there is
  $\t\colon Y\to Y'$ in $S$ such that $\t\a=\t\b$.
\end{enumerate}

If $S$ admits a calculus of left fractions, then the morphisms in
$\D[S^{-1}]$ are of the form $\s^{-1}\a$ given by a pair of morphisms
$X\stackrel{\a}\to Y'\stackrel{\s}\leftarrow Y$ in $\D$ with
$\s\in S$, where we identify a morphism in $\D$ with its image under
the canonical functor $\D\to \D[S^{-1}]$.  For pairs $(\a_1,\s_1)$ and
$(\a_2,\s_2)$ we have $\s_1^{-1}\a_1=\s_2^{-1}\a_2$ in $\D[S^{-1}]$ if
and only if there exists a commutative diagram
\[
\begin{tikzcd}
&Y_1\arrow{d}\\ X\arrow{r}{\a}\arrow{ru}{\a_1}\arrow[swap]{rd}{\a_2}&Y'&
Y\arrow[swap]{lu}{\s_1}\arrow{ld}{\s_2}\arrow[swap]{l}{\s}\\
&Y_2\arrow{u}
\end{tikzcd}
\]
with $\s$ in
$S$; see  \cite{GZ1967}.

\begin{lem}\label{le:LF}
  The eventually invertible morphisms in $\Cau(\bbN,\C)$ admit a
  calculus of left fractions.
\end{lem}

We need some preparations for the proof of this lemma.  Given functors
$f\colon\bbN\to\bbN$ and $X\colon \bbN\to\C$, let $X_f$ denote the
composite $X\circ f$. Call $f$ \emph{cofinal} if $n\le f(n)$ for all
$n\in\bbN$. In this case there is a natural morphism $f_X\colon X\to X_f$.

A straightforward computation of filtered colimits in $\Set$ yields
the following.

\begin{lem}\label{le:cofinal}
   Let $X,Y\colon \bbN\to \C$ be functors.
\begin{enumerate}
\item Given a morphism $\p\colon \widetilde X\to \widetilde Y$, there
  exists a cofinal $f\colon \bbN\to\bbN$ and a morphism
  $\a\colon X\to Y_f$ such that $\widetilde{f_Y}\p=\widetilde\a$.
\item Given morphisms $\a,\b\colon X\to Y$ such that
  $\widetilde\a=\widetilde\b$, there exists a cofinal
  $f\colon \bbN\to\bbN$ such $f_Y\a=f_Y\b$.
\end{enumerate}  
\end{lem}
\begin{proof}
  (1) Fix $n\in\bbN$ and suppose $f$ has been defined for all
  $m<n$. There exists $n'\ge \max(n, f(n-1))$ and  $\a_n\colon X_n\to
  Y_{n'}$ such that the composite
  \[\Hom(X_{n},X_{n})\to\colim_i\Hom(X_n,X_i)\xto{\p}\colim_i\Hom(X_n,Y_i)\]
  maps $\id_{X_n}$ to the image of $\a_n$ under
  $\Hom(X_{n},Y_{n'})\to\colim_i\Hom(X_n,Y_i)$. We set  $f(n)=n'$ and
  can make this choice consistent such that the $\a_n$
  yield a morphism $X\to Y_f$.

  (2) Fix $n\in\bbN$ and suppose $f$ has been defined for all
  $m<n$. The fact that $\a$ and $\b$ induce the same map
  $\colim_i\Hom(X_n,X_i)\to\colim_i\Hom(X_n,Y_i)$ yields $n'\ge \max(n,f(n-1))$
  such that $\Hom(X_n,Y_n)\to \Hom(X_n,Y_{n'})$ maps $\a_n$ and $\b_n$
  to the same element. Then set $f(n)=n'$.
\end{proof}

Let $S$ denote the class of eventually invertible morphisms. If
$X\colon\bbN\to\C$ is Cauchy and $f\colon\bbN\to\bbN$ is cofinal, then
$X_f$ is Cauchy and the canonical morphism $f_X\colon X\to X_f$ is in
$S$. Moreover, for any $\s\colon X\to Y$ in $S$ there exists a cofinal
$f\colon\bbN\to\bbN$ and a morphism $\s'\colon Y\to X_f$ such that
$\s'\s=f_X$. This follows by applying Lemma~\ref{le:cofinal} to ${\widetilde\s}^{-1}$.

\begin{proof}[Proof of Lemma~\ref{le:LF}]
  The condition (LF1) is clear.  To check (LF2) fix a pair of
  morphisms $X'\xleftarrow{\s}X\xto{\a} Y$ with $\s\in S$. Choose
  $\s'\colon X'\to X_f$ such that $\s'\s=f_X$. Then we obtain the
  following commutative square with $f_Y\in S$.
\[\begin{tikzcd}
X\arrow{rr}{\a}\arrow[swap]{d}{\s}&&
Y\arrow{d}{f_Y}\\
X'\arrow{r}{\s'}&X_f\arrow{r}{\a_f}&Y_f
\end{tikzcd}
\]

To check (LF3) fix a pair of morphisms $\a,\b\colon X\to Y$. Let
$\s\colon X'\to X$ in $S$ such that $\a\s=\b\s$. This implies
$\widetilde\a=\widetilde\b$, and it follows from Lemma~\ref{le:cofinal}
that there exists a cofinal $f\colon\bbN\to\bbN$ such that
$f_Y\a=f_Y\b$.  Note that $f_Y\in S$.
\end{proof}

\begin{prop}\label{pr:completion}
  The canonical functor $\widehat\C\to\Fun(\C^\op,\Set)$ is fully
  faithful; it identifies $\widehat\C$ with the colimits of
  sequences of representable functors that correspond to Cauchy
  sequences in $\C$. Also, the canonical functor
  $\C\to\widehat\C$ is fully faithful.
\end{prop}
\begin{proof}
  We use the fact that the class $S$ of eventually invertible
  morphisms in $\Cau(\bbN,\C)$ admits a calculus of left
  fractions. Then every morphism in $\widehat\C$ is of the form
  $\s^{-1}\a$ given by a pair of morphisms
  $X\xto{\a}Y'\xleftarrow{\s} Y$ in $\Cau(\bbN,\C)$ with $\s\in S$.
  
  Fix Cauchy sequences $X,Y\colon\bbN\to\C$. We need to show that the
  canonical map
  \[\Hom(X,Y)\lto \Hom(\widetilde X,\widetilde Y),\quad
    \s^{-1}\a\mapsto\widetilde\s^{-1}\widetilde\a,\] is a bijection.  The
  map is surjective, because Lemma~\ref{le:cofinal} yields for any
  morphism $\p\colon \widetilde X\to\widetilde Y$ a cofinal
  $f\colon\bbN\to\bbN$ and $\a\colon X\to Y_f$ such that
  $\p=\widetilde{f_Y}^{-1}\widetilde\a$. Now fix a pair of morphisms
  $\s_1^{-1}\a_1$ and $\s_2^{-1}\a_2$ in $\widehat\C$ such that
  $\widetilde{\s_1}^{-1}\widetilde{\a_1}=\widetilde{\s_2}^{-1}\widetilde{\a_2}$. Use
  (LF2) to complete $\s_1$ and $\s_2$ to a commutative diagram
 \[
\begin{tikzcd}
  Y\arrow{r}{\s_2}\arrow[swap]{d}{\s_1}&
  Y''\arrow{d}{\t_2}\\
  Y'\arrow{r}{\t_1}&Z
\end{tikzcd}
\]  
with $\t_1,\t_2\in S$. Then we have
$\widetilde{\t_1\a_1}=\widetilde{\t_2\a_2}$ and there exists a cofinal
$f\colon\bbN\to\bbN$ such that $f_Z\t_1\a_1=f_Z\t_2\a_2$. This implies
$\s_1^{-1}\a_1=\s_2^{-1}\a_2$ in $\widehat\C$.

  Colimits in $\Fun(\C^\op,\Set)$ are computed `pointwise'. Thus for
  $X$ in $\Fun(\bbN,\C)$ the functor $\widetilde X$ is the colimit of
  the sequence \[\Hom(-,X_0)\to \Hom(-,X_1)\to \Hom(-,X_2)\to\cdots\]
  in $\Fun(\C^\op,\Set)$. It follows that $\widehat\C$ identifies with
  the colimits of sequences of representable functors that correspond to Cauchy
  sequences in $\C$.

  Finally, the canonical functor $\C\to\widehat\C$ is fully faithful,
  since the composition with $\widehat\C\to\Fun(\C^\op,\Set)$ is fully
  faithful by Yoneda's lemma.
\end{proof}

\begin{cor}\label{co:completion}
   For $X,Y\in\widehat\C$ we have a natural bijection
  \[\Hom(X,Y)\xto{_\sim}\lim_i\colim_j\Hom(X_i,Y_j).\]
\end{cor}

\begin{proof}
Combining  Proposition~\ref{pr:completion} and Yoneda's lemma, we have
\begin{align*}
  \Hom(X,Y)&\cong\Hom(\colim_i\Hom(-,X_i), \colim_j\Hom(-,Y_j))\\
           &\cong \lim_i\Hom(\Hom(-,X_i), \colim_j\Hom(-,Y_j))\\
           &\cong \lim_i\colim_j\Hom(X_i,Y_j).\qedhere
\end{align*}  
\end{proof}

Call $\C$ \emph{sequentially complete} if every Cauchy sequence
in $\C$ has a colimit in $\C$. Clearly, $\C$ is sequentially complete if and only
if the canonical functor $\C\to\widehat\C$ is an equivalence.
We do not know whether $\widehat\C$ is always sequentially complete.

\begin{rem}\label{re:Ind}
  From Proposition~\ref{pr:completion} it follows that the sequential
  completion of $\C$ identifies with a full subcategory of the
  ind-completion $\Ind(\C)$ in the sense of \cite[\S8]{GV}.
\end{rem}

\begin{rem}
  Let $F\colon\C\to\D$ be a functor.

(1) Suppose that $F$  admits a left adjoint. Then
  $F$ preserves Cauchy sequences and induces therefore a functor
  $\widehat F\colon\widehat\C\to\widehat\D$ such that
\begin{equation}\label{eq:functor}
  \widehat F(X)=\colim_iF(X_i) \quad\textrm{for} \quad X\in\widehat
  C.
\end{equation}

(2) Suppose that $\D$ admits filtered colimits. Then $F$ extends via
\eqref{eq:functor} to a functor $\widehat\C\to\D$.

  (3) If $(F,G)$ is an adjoint pair of functors
  that preserve Cauchy sequences, then  $(\widehat F,\widehat G)$ is
  an adjoint pair since
  \begin{align*}
    \Hom(\widehat F(X),Y)&\cong\Hom(\colim_iF(X_i), \colim_j Y_j)\\
                         &\cong \lim_i\colim_j\Hom(F(X_i),Y_j)\\
                         &\cong \lim_i\colim_j\Hom(X_i,G(Y_j))\\
                         &\cong \Hom(X,\widehat G(Y)).
\end{align*}  
\end{rem}

The notion of a Cauchy sequence goes back to work of Bolzano and
Cauchy (providing a criterion for convergence), while the construction
of the real numbers from the rationals via equivalence classes of
Cauchy sequences is due to Cantor and M\'eray
\cite{Ca1872,Me1869}. The sequential completion generalises this
construction.

\begin{exm}
  View the rational numbers $\bbQ=(\bbQ,\le)$ with the usual ordering
  as a category and let $\bbR_\infty=(\bbR\cup\{\infty\},\le)$.
  Taking a Cauchy sequence to its limit yields a functor
    \[\widehat\bbQ\longtwoheadrightarrow\bbR_\infty,\quad (x_i)\mapsto
      \lim_{i\to\infty}x_i.\] For $x\in\bbR_\infty$, there are
    precisely two isomorphism classes of objects in $\widehat\bbQ$
    with limit $x$ when $x$ is rational (depending on whether the
    sequence is eventually constant or not); otherwise there is
    precisely one isomorphism class in $\widehat\bbQ$ with limit
    $x$.\footnote{I am grateful to Zhenqiang Zhou for pointing out an
      error in the first version of this paper.}
\end{exm}
\begin{proof}
  A Cauchy sequence $x\in\Cau(\bbN,\bbQ)$ is by definition a sequence
  $x_0\le x_1\le x_2\le \cdots$ of rational numbers that is either
  bounded, so converges to $\bar x\in\bbR$, or it is unbounded and we
  set $\bar x=\infty$. Given a morphism $x\to y$ in $\Cau(\bbN,\bbQ)$
  that is eventually invertible, we have $\bar x=\bar y$. Conversely,
  if $\bar x=\bar y$, then we define $u\in\Cau(\bbN,\bbQ)$ by
  $u_i=\min(x_i,y_i)$ and have morphisms $x\leftarrow u\to y$. It is
  easily checked that $u\to x$ is eventually invertible, except when
  $x$ is eventually constant and $y$ is not.  Thus the assignment
  $x\mapsto \bar x$ yields the desired functor $\widehat\bbQ\twoheadrightarrow\bbR_\infty$.
\end{proof}

Let  $\two=\{0<1\}$ denote the poset consisting of two elements.
For any category $\C$, the category 
of morphisms in $\C$ identifies with $\C^{\two}=\Fun(\two,\C)$.

\begin{exm}\label{example:morphism-category}
  Let $\C$ be a category that admits an initial object and set $\D=\C^\two$. Then there is a
  canonical equivalence $\widehat\C^\two\xto{_\sim}\widehat{\D}$.
\end{exm}
\begin{proof}
  The equivalence
  \[F\colon \Fun(\two,\Fun(\bbN,\C))\xto{_\sim}\Fun(\two\times\bbN,\C)
    \xto{_\sim}\Fun(\bbN,\Fun(\two,\C))
  \]
  restricts to an equivalence
   \[F_0\colon \Fun(\two,\Cau(\bbN,\C))
    \xto{_\sim}\Cau(\bbN,\Fun(\two,\C)).
  \]
  In order to see this, let $\p\colon X\to Y$ be a morphism in
  $\Fun(\bbN,\C)$ and $\a\colon C\to D$ a morphism in $\C$. Suppose
  $X$ and $Y$ are in $\Cau(\bbN,\C)$. Thus there is $n\in\bbN$ such
  that $\Hom(C,X_i)\xto{_\sim}\Hom(C,X_{i+1})$ and
  $\Hom(D,Y_i)\xto{_\sim}\Hom(D,Y_{i+1})$ for $i\ge n$. Then
  $\Hom(\a,\p_i)\xto{_\sim}\Hom(\a,\p_{i+1})$ for $i\ge n$. Thus
  $F(\p)$ is in $\Cau(\bbN,\Fun(\two,\C))$. On the other hand, if
  $F(\p)$ is in $\Cau(\bbN,\Fun(\two,\C))$, then we choose $\a=\id_C$
  and $\a\colon I\to D$ ($I$ the initial object in $\C$) to see that
  $X$ and $Y$ are in $\Cau(\bbN,\C)$. It is easily checked that $F_0$
  induces a functor $\widehat\C^\two\to\widehat{\C^\two}$, and we
  claim that it is an equivalence. Recall from Remark~\ref{re:Ind}
  that there is a canonical embedding $\widehat\C\to\Ind(\C)$. So it
  remains to use the fact that $\Ind(\C)^\two\xto{_\sim}\Ind(\C^\two)$,
  which follows from Propositions~8.8.2 and 8.8.5 in \cite{GV}.
\end{proof}

Let us generalise the definition of the completion $\widehat\C$, because later on we
need to modify the underlying choice of Cauchy sequences.  We fix a
class $\X$ of objects in $\Fun(\bbN,\C)$. The \emph{completion of $\C$
  with respect to $\X$} is the category $\widehat{\C_\X}$ with class
of objects $\X$ and
\[\Hom(X,Y)=\lim_i\colim_j\Hom(X_i,Y_j)\quad\text{for}\quad
  X,Y\in\X.\] It follows from the definition that the assignment
$X\mapsto\colim_i\Hom(-,X_i)$ induces a fully faithful functor
$\widehat{\C_\X}\to\Fun(\C^\op,\Set)$. Clearly, $\widehat{\C}$
identifies with $\widehat{\C_\X}$ when $\X$ equals the class of Cauchy
sequences, by Corollary~\ref{co:completion}.

\begin{exm}\label{ex:countable-envelope}
 Let $\C$ be an exact
category and let $\X$ denote the class of sequences $X$ such that each
$X_i\to X_{i+1}$ is an admissible monomorphism. Then
$\C\tilde{\phantom{e}}:=\widehat{\C_\X}$ admits a canonical exact
structure and is called \emph{countable envelope} of $\C$
\cite[Appendix~B]{Ke1990}.
\end{exm}

\section{The sequential completion of an abelian category}
\label{se:abelian}

Let $\C$ be an additive category.

\begin{lem}
  The sequential completion
$\widehat\C$ is an additive category and the canonical functor
$\C\to\widehat\C$ is additive.
\end{lem}
\begin{proof}
The assertion  follows from the fact that
$\Cau(\bbN,\C)$ is additive and that the eventually invertible
morphisms admit a calculus of left fractions \cite[I.3.3]{GZ1967}.
\end{proof}

It follows that the assignment $X\mapsto\widetilde X$ yields a fully
faithful additive functor $\widehat\C\to\Add(\C^\op,\Ab)$ into the
category of additive functors $\C^\op\to\Ab$. 

\begin{lem} 
  If $\C$ admits kernels,
  then $\widehat\C$ admits kernels and $\C\to\widehat\C$ is left
  exact.
\end{lem}
\begin{proof}
  A morphism $X\to Y$ in $\widehat\C$ is up to isomorphism given by a
  morphism $\p\colon X\to Y$ in $\Cau(\bbN,\C)$. Then
  $K:=(\Ker\p_i)_{i\ge 0}$ is a Cauchy sequence, and this yields the
  kernel in $\widehat\C$, because the sequence
  $0\to\widetilde K\to\widetilde X \to\widetilde Y$ is exact in
  $\Add(\C^\op,\Ab)$.
\end{proof}

Let $\A$ be an abelian category. We write $\fl\A$ for the full
subcategory of objects having finite composition length, and $\art\A$
denotes the full subcategory of artinian objects. For $X\in\A$ the
\emph{socle} $\soc X$ is the sum of all simple subobjects. One
defines inductively $\soc^n X\subseteq X$ for $n\ge 0$ by setting $\soc^0 X=0$,
and $\soc^{n+1}X$ is given by the exact sequence
\[0\lto\soc^n X\lto\soc^{n+1} X\lto\soc (X/\soc^n X)\lto 0.\]

\begin{exm}\label{ex:finite-length}
  Let $\A=\Mod\La$ be the module category of a commutative noetherian
  local ring $\La$. Then the sequential completion of $\fl\A$
  identifies with $\art\A$.
\end{exm}

\begin{proof}
  Set $\C=\fl\A$.  It is well known that a $\La$-module $X$ is
  artinian if and only if its socle has finite length.  In that case the socle
  series $(\soc^iX)_{i\ge 0}$ of $X$ yields a Cauchy sequence in $\C$
  with $\colim_i(\soc^iX)=X$.

  Now let $X\in\widehat\C$. The assignment
  $X\mapsto \bar X:=\colim_i X_i$ yields a fully faithful functor
  $\widehat\C\to\A$. Let $S$ denote the unique (up to isomorphism)
  simple object in $\A$. Then $\soc\bar X$ has
  finite length, since $\soc\bar X\cong\Hom(S,\bar
  X)\cong\colim_i\Hom(S, X_i)$. Thus $\bar X$ is artinian.
\end{proof}

The preceding example suggests a general criterion such that the
sequential completion of an abelian category is  abelian.

Let us fix a \emph{length category} $\C$. Thus $\C$ is an abelian
category and every object has finite length.  We call $\C$
\emph{ind-artinian} if
\begin{enumerate}
  \item $\C$ has only finitely many isomorphism classes of simple objects, 
  \item $\C$ is \emph{right Ext-finite}, that is, for every pair of
    simple objects $S$ and $T$ the $\End(S)$-module $\Ext^1(S,T)$ has
    finite length, and
  \item $\C$ satisfies the descending chain condition on subobjects of socle stable sequences in $\C$.
\end{enumerate}
  Here, we consider sequences $X=(X_i\to X_{i+1})_{i\ge 0}$ of morphisms in $\C$, so functors $(\bbN,\le)\to\C$, and $X$ is \emph{socle stable} if $X_i\xto{\sim} \soc^i X_j$ for all $j\ge i$. A subobject $X\subseteq Y$ is given by a morphism of functors $X\to Y$ such that $X_i\to Y_i$ is a monomorphism for all $i\ge 0$.

\begin{prop}\label{pr:ind-artinian}
  Let $\A$ be a Grothendieck category with a fully faithful functor
  $\C\hookrightarrow\A$ that identifies $\C$ with the full subcategory
  of finite length objects in $\A$. Suppose that every object in $\A$
  is the union of its finite length subobjects. Then the following are
  equivalent:
  \begin{enumerate}
  \item The category $\C$ is ind-artinian.
  \item The category $\C$ has only finitely many isomorphism classes
    of simple objects, and an object in $\A$ is artinian if its socle
    has finite length.
 \item The category $\A$  admits an artinian cogenerator.
\end{enumerate}
In this case an object in $\A$ is artinian if and only if it is the
colimit of a Cauchy sequence in $\C$.
\end{prop}

\begin{proof}
  Let us begin with the observation that for every artinian object
  $X\in\A$ the socle series $(\soc^i X)_{i\ge 0}$ is a Cauchy sequence
  in $\C$ with colimit $X$. To see this, note that $X_i:=\soc^i X\in\C$ for
  all $i$ since $X_i/X_{i-1}$ is semisimple and artinian, so
  of finite length. Furthermore, each object $C\in\C$ satisfies
  $\soc^nC=C$ for some $n$, and then every morphism $C\to X$ factors
  through $X_n$. Thus $\Hom(C,X_i)\xto{_\sim}\Hom(C,X_{i+1})$ for
  all $i\ge n$, and $X=\bigcup_i\soc^i X$.

  (1) $\Rightarrow$ (2): Let $X\in\A$ and suppose that $\soc X$ has
  finite length.  An injective envelope $X\to E$ induces an
  isomorphism $\soc X\xto{_\sim}\soc E$.  So we may assume that $X$ is
  injective.  Set $X_i:=\soc^iX$ for $i\ge 0$. The assumption on $\C$
  implies that $X_n$ has finite length for all $n > 0$. This follows
  by induction from the defining exact sequence for $\soc^n$ as
  follows.  Let $S=\bigoplus_i S_i$ be the direct sum of a
  representative set of simple objects.  For $n>0$ we have an
  isomorphism of
  $\End(S)$-modules\[\Hom(S,X/X_n)\cong\Ext^1(S,X_n),\] and
  their length equals the length of $\soc (X/X_n)$. Thus
  $X_n\in\C$ for all $n>0$.

 The sequence $(X_i\to X_{i+1})_{i\ge 0}$ is socle stable, and the
  subobjects $U\subseteq X$ identify with subobjects of this socle
  stable sequence by taking $U$ to the sequence
  $(U_i\to U_{i+1})_{i\ge 0}$ given by $U_i:=\soc^iU=U\cap X_i$. The
  inverse map takes a subsequence $(V_i\to V_{i+1})_{i\ge 0}$ given by
  subobjects $V_i\subseteq X_i$ to $\bigcup_iV_i\subseteq X$.  Thus
  $X$ is artinian.

  (2) $\Rightarrow$ (3): Choose an injective envelope $E=E(S)$ in $\A$
  for the direct sum $S=\bigoplus_i S_i$ of a representative set of
  simple objects. Then $E$ is an injective cogenerator which is
  artinian since $\soc E\cong S$.
 
  (3) $\Rightarrow$ (1): Choose an artinian cogenerator $E$ of $\A$,
  and we may assume $E$ is injective since for each simple $S$ the
  injective envelope $E(S)$ is a direct summand of $E$.  The number of
  isoclasses of simple objects is bounded by the length of $\soc E$
  and is therefore finite. Fix simple objects $S,T\in\A$ and choose a
  monomorphism $T\to E$. Then $\Hom(S,E/T)\cong \Ext^1(S,T)$ and the
  length of $\Ext^1(S,T)$ as $\End(S)$-module is bounded by the length
  of $\soc (E/T)$ which is finite since $E/T$ is artinian. It follows
  that $\C$ is right ext-finite. Now fix a socle stable sequence
  $(X_i\to X_{i+1})_{i\ge 0}$ in $\C$ and set $X=\colim_i X_i$. Note
  that $X_i\cong \soc^iX$ for all $i\ge 0$.  Then $X$ embeds into the
  injective envelope $E(X_1)$ and is therefore artinian. Subobjects of
  $(X_i\to X_{i+1})_{i\ge 0}$ correspond to subobjects of $X$ by the
  first part of the proof. Thus $\C$ is ind-artinian.

  It remains to establish the last assertion. We have already seen
  that any artinian object is the colimit of a Cauchy sequence in
  $\C$. Conversely, let $X\in\A$ be the colimit of a Cauchy sequence
  $(X_i)_{i\ge 0}$ in $\C$ and let $n\in\bbN$ such that for every simple
  object $S$ we have $\Hom(S,X_i)\xto{_\sim}\Hom(S,X_{i+1})$ for all
  $i\ge n$. Then $\soc X=\soc X_n$ is in $\C$, and therefore $X$ is
  artinian.
\end{proof}

Recall that an abelian category satisfies the (AB5) condition if for
every directed set of subobjects $(A_i)_{i\in I}$ of an object $A$ and
$B\subseteq A$ one has
\[\left(\sum_{i\in I}A_i\right)\cap B=\sum_{i\in I}(A_i\cap B).\]

\begin{cor}
  Let $\C$ be a length category and suppose $\C$ is ind-artinian. Then
  $\widehat\C$ is an abelian category with injective envelopes,
  satisfying the \emph{(AB5)} condition, and every object is
  artinian. Moreover, the canonical functor $\C\to\widehat\C$ induces
  an equivalence $\C\iso\fl\widehat\C$.
\end{cor}
\begin{proof}
  We embed $\C$ into a Grothendieck category via the functor
  \[\C\lto\A:=\Lex(\C^\op,\Ab),\quad X\mapsto\Hom(-,X),\]
  where $\Lex(\C^\op,\Ab)$ denotes the category of left exact functors
  $\C^\op\to\Ab$; see \cite{Ga1962}. Then $\widehat\C$ identifies with
  a subcategory of $\A$ via Proposition~\ref{pr:completion}, and
  Proposition~\ref{pr:ind-artinian} implies that $\widehat\C=\art\A$. From
  this the assertion follows.
\end{proof}

\begin{exm}\label{ex:noeth-alg}
  Let $\La$ be a noetherian algebra over a complete local ring and set
  $\A=\Mod\La$. Then $\fl\A$  is ind-artinian and $\widehat{\fl\A}$
  identifies with $\art\A$.
\end{exm}
\begin{proof}
  There are only finite many simple $\La$-modules, up to isomorphism,
  and their injective envelopes are artinian. Thus $\fl\A$ is
  ind-artinian by Proposition~\ref{pr:ind-artinian}, and therefore $\widehat{\fl\A}$
  identifies with $\art\A$.
\end{proof}

\begin{exm}
  Let $\La$ be a ring and $\C\subseteq \Mod\La$ a full subcategory of
  its module category that contains $\La$. Then $\C$ is sequentially
  complete.
\end{exm}
\begin{proof}
  Let $X\in\Cau(\bbN,\C)$. We have $\Hom(\La,X_i)\cong X_i$,  and therefore
$X_i\xto{_\sim}X_{i+1}$ for $i\gg 0$. Thus  $\colim_i X_i$ belongs to $\C$.
\end{proof}

\section{The sequential completion of a triangulated category}
\label{s:seq-completion}

Let $\T$ be a triangulated category and suppose that countable coproducts
exist in $\T$. Let
\[
\begin{tikzcd}
X_0\arrow{r}{\p_0}&X_1\arrow{r}{\p_1}&
  X_2\arrow{r}{\p_2}&\cdots
\end{tikzcd}\] be a sequence of morphisms in $\T$. A \emph{homotopy
colimit} of this sequence is by definition an object $X$ that occurs
in an exact triangle
\[\begin{tikzcd}
    \Si^{-1}X\arrow{r}&\coprod_{i\ge
      0}X_i\arrow{r}{\id-\p}&\coprod_{i\ge 0}X_i\arrow{r}& X.
  \end{tikzcd}\] We write $\hocolim_i X_i$ for $X$ and observe that a
homotopy colimit is unique up to a (non-unique) isomorphism \cite{BN1993}.

Recall that an object $C$ in $\T$ is \emph{compact} if $\Hom(C,-)$
preserves all coproducts. A morphism $X\to Y$ is \emph{phantom} if any
composition $C\to X\to Y$ with $C$ compact is zero
\cite{Ch1998,CS1998}. The phantom morphisms form an ideal and we write
$\Ph(X,Y)$ for the subgroup of all phantoms in $\Hom(X,Y)$. Let us
denote by $\T/\Ph$ the additive category which is obtained from $\T$
by annihilating all phantom morphisms.

\begin{lem}\label{le:hocolim}
  Let $C\in\T$ be compact. Any sequence $X_0\to X_1\to X_2\to\cdots$
  in $\T$ induces an isomorphism
  \[\colim_i\Hom(C,X_i)\xto{\ _\sim\ }\Hom(C,\hocolim_i X_i).\]
\end{lem}
\begin{proof}
  See \cite[\S5.1]{Ke1994} or \cite[Lemma~1.5]{Ne1992}.
\end{proof}

Recall that for any sequence
$\cdots\to A_2\xto{\p_2} A_1\xto{\p_1} A_0$ of maps between abelian
groups the inverse limit and its first derived functor are given by
the exact sequence
\[\begin{tikzcd}
  0\lto\lim_iA_i\lto  \prod_{i\ge 0}A_i\arrow{r}{\id-\p}&\prod_{i\ge 0}A_i\arrow{r}&
    \lim^1_iA_i\arrow{r}& 0.
  \end{tikzcd}\]

The following result goes back to work of Milnor \cite{Mi1962} and was
later extended by several authors, for instance in \cite{Ch1998,CS1998}. 

\begin{lem}\label{le:phantom}
  Let $X=\hocolim_i X_i$ be a homotopy colimit in $\T$ such that
  each $X_i$ is a coproduct of compact objects. Then we have for any
  $Y$ in $\T$ a natural exact sequence
\[0\lto\Ph(X,Y)\lto\Hom(X,Y)\lto\lim_i\Hom(X_i,Y)\lto 0\]
and an isomorphism
\[\Ph(X,\Si Y)\cong \sideset{}{^{1}}\lim_i\Hom(X_i,Y).\]
\end{lem}
\begin{proof}
  Apply $\Hom(-,Y)$ to the exact triangle defining $\hocolim_i X_i$
  and use that a morphism $X\to Y$ is phantom if and only if it
  factors through the canonical morphism
  $X\to\coprod_{i\ge 0}\Si X_i$.
\end{proof}

Let $\C\subseteq\T$ be a full additive subcategory consisting of
compact objects and consider the restricted Yoneda functor
\[\T\lto\Add(\C^\op,\Ab),\quad X\mapsto h_X:=\Hom(-,X)|_\C.\]
Note that for any sequence $X_0\to X_1\to X_2\to\cdots$ in $\C$ we
have by Lemma~\ref{le:hocolim}
\begin{equation}\label{eq:colim}
  \widetilde X=\colim_i\Hom(-,X_i)=h_{\hocolim_i X_i}.
\end{equation}

\begin{lem}\label{le:yoneda}
  Let $X=\hocolim_i X_i$ be a homotopy colimit in $\T$ such that
  each $X_i$ is a coproduct of objects in $\C$. Then we have for any
  $Y$ in $\T$ a natural isomorphism
  \[\frac{\Hom(X,Y)}{\Ph(X,Y)}\xto{\ _\sim\ }\Hom(h_X,h_Y).\]
  \end{lem}
  \begin{proof}
    Using the preceding lemmas, we have
    \begin{align*}
      \frac{\Hom(X,Y)}{\Ph(X,Y)}&\cong\lim_i\Hom(X_i,Y)\\
                                &\cong\lim_i\Hom(h_{X_i},h_{Y})\\
                                &\cong\Hom(\colim_ih_{X_i},h_{Y})\\
      &\cong\Hom(h_X,h_Y).\qedhere
      \end{align*}
 \end{proof}

\begin{prop}\label{pr:completion-phantom}
  Let $\C\subseteq\T$ be a full additive subcategory consisting of
  compact objects. Taking a sequence $X_0\to X_1\to X_2\to \cdots$ in
  $\C$ to its homotopy colimit induces a fully faithful functor
  $\widehat{\C}\to\T/\Ph$.
\end{prop}
\begin{proof}
  We have the functor
\[\widehat\C\lto\Add(\C^\op,\Ab),\quad X\mapsto \widetilde X,\]
which is  fully faithful by Proposition~\ref{pr:completion}.
Now combine \eqref{eq:colim} and Lemma~\ref{le:yoneda}.
\end{proof} 

\begin{defn}
  Let $\C$ be a triangulated category and $\X$ a class of sequences
  $(X_i)_{i\ge 0}$ in $\C$ that is stable under suspensions, i.e.\
  $(\Si^nX_i)_{i\ge 0}$ is in $\X$ for all $n\in\bbZ$. We say that
  $\X$ is \emph{phantomless} if for any pair of sequences $X,Y$ in
  $\X$ we have
\begin{equation}\label{eq:ph-less}
  \sideset{}{^{1}}\lim_i\colim_j\Hom(X_i,Y_j)=0.
\end{equation}
\end{defn}

The following lemma justifies the term `phantomless'.

\begin{lem}\label{le:phantomless}
  Let $\C\subseteq\T$ be a full triangulated subcategory consisting of
  compact objects and $\X$ a class of sequences $(X_i)_{i\ge 0}$ in
  $\C$ that is stable under suspensions. Consider the full subcategory
  \[\D:=\{\hocolim_i X_i\in\T\mid X\in\X\}\subseteq\T.\] Then the following are
  equivalent:
 \begin{enumerate}
 \item The class $\X$ is phantomless.
 \item We have $\Ph(U,V)=0$ for all $U,V\in\D$.
 \item The functor $\D\to\Add(\C^\op,\Ab)$ taking $X$ to
   $\Hom(-,X)|_\C$ is fully faithful.
\item The assignment $X\mapsto\hocolim_i X_i$ yields an
  equivalence $\widehat{\C_\X}\xto{_\sim}\D$.   
\end{enumerate}
In this case the homotopy colimit of a Cauchy sequence in
$\C$ is actually a colimit in $\D$, provided that $\X$ contains all
constant sequences  consisting of identities only.
\end{lem}
\begin{proof}
  (1) $\Leftrightarrow$ (2): Combine  Lemmas~\ref{le:hocolim} and
  \ref{le:phantom}.

  (2) $\Leftrightarrow$ (3): Apply Lemma~\ref{le:yoneda}.

  (2) $\Leftrightarrow$ (4): Apply
  Proposition~\ref{pr:completion-phantom}.

  The final assertion follows from the identity \eqref{eq:colim} for
  any sequence $X$ in $\C$, since $\D$ identifies with a full
  subcategory of $\Add(\C^\op,\Ab)$.
\end{proof}

Recall that a triangulated category is \emph{algebraic} if it is
triangle equivalent to the stable category $\St(\A)$ of a Frobenius
category $\A$. 
A morphism between exact triangles
\[
\begin{tikzcd}
  X\arrow{r}\arrow{d}&Y\arrow{r}\arrow{d}&
  Z\arrow{r}\arrow{d}&\Si
  X\arrow{d}\\
  X'\arrow{r}&Y'\arrow{r}&Z'\arrow{r}&\Si X'
\end{tikzcd}
\]
in $\St(\A)$ will be called \emph{coherent} if it can be lifted to
a morphism
\[
\begin{tikzcd}
  0\arrow{r}&\tilde X\arrow{r}\arrow{d}& \tilde
  Y\arrow{r}\arrow{d}&
  \tilde  Z\arrow{r}\arrow{d}&0\\
  0\arrow{r}&\tilde X'\arrow{r}&\tilde
  Y'\arrow{r}&\tilde Z'\arrow{r}&0
\end{tikzcd}
\]
between  exact sequences in $\A$ (so that the canonical functor
$\A\to\St(\A)$ maps the second to the first diagram). Note that
any commutative diagram
\[
\begin{tikzcd}
  X\arrow{r}\arrow{d}&Y\arrow{d}\\
  X'\arrow{r}&Y'
\end{tikzcd}
\]
in $\St(\A)$ can be completed to a coherent morphism of exact
triangles as above.

The following theorem establishes a triangulated structure for the
sequential completion of a triangulated category $\C$. Let us stress
that we use a relative version of this result for our applications, as
explained in Remark~\ref{re:bounded} below; it depends on the
choice of a class $\X$ of sequences in $\C$ which is phantomless.

\begin{thm}\label{th:tria-completion}
  Let $\C$ be an algebraic triangulated category, viewed as a
  full subcategory of its sequential completion $\widehat\C$. Suppose
  that the class of Cauchy sequences is phantomless. Then $\widehat\C$
  admits a unique triangulated structure such that the exact triangles
  are precisely the ones isomorphic to colimits of Cauchy sequences
  that are given by coherent morphisms of exact triangles in $\C$.
\end{thm}

Let us spell out the triangulated structure for $\widehat\C$. Fix a
sequence of coherent morphisms $\eta_0\to\eta_1\to\eta_2\to\cdots$ of
exact triangles
  \[\eta_i\colon X_i\lto Y_i\lto Z_i\lto\Si X_i\] in $\C$ and suppose
  that it is also a sequence of morphisms $X\to Y\to Z\to \Si X$ of Cauchy
  sequences in $\C$. This identifies with the sequence
  \[\colim_i X_i \lto \colim_i Y_i\lto \colim_i Z_i\lto \colim_i\Si X_i\]
in $\widehat \C$, and the exact triangles in  $\widehat \C$ are
precisely sequences of morphisms that are isomorphic to sequences of
the above form.

Theorem~\ref{th:morphic-completion} provides a substantial
generalisation, from algebraic triangulated categories to triangulated
categories with a morphic enhancement. Moreover, in some interesting
cases the morphic enhancement extends to a morphic enhancement of the
completion; see Section~\ref{ss:completion-morphic-enrich}. 

\begin{proof}
  The proof is given in several steps.

  (1) The assumption on $\C$ to be algebraic implies that $\C$
  identifies with the stable category $\St(\A)$ of a Frobenius
  category $\A$. Let $\A\tilde{\phantom{e}}$ denote the countable
  envelope of $\A$ which is a Frobenius category containing $\A$ as a full
  exact subcategory; see Example~\ref{ex:countable-envelope}. Then $\C$
  identifies with a full triangulated subcategory of compact objects
  of $\T:=\St(\A\tilde{\phantom{e}})$.

 For any sequence of coherent morphisms
  $\eta_0\to\eta_1\to\eta_2\to\cdots$ of exact triangles
  \[\eta_i\colon X_i\lto Y_i\lto Z_i\lto\Si X_i\] in $\C$ there is in
  $\T$ an induced exact
  triangle
  \begin{equation}\label{eq:tria-colim}
    \hocolim_i X_i\lto \hocolim_i Y_i\lto \hocolim_i Z_i\lto\Si(\hocolim_i
    X_i).
  \end{equation}
  Let us sketch the argument. We can lift the
  sequence $(\eta_i)_{i\ge 0}$ to a sequence of
  exact sequences 
   \[\tilde\eta_i\colon 0\lto \tilde X_i\lto \tilde Y_i\lto \tilde
    Z_i\lto 0\] in $\A$ and obtain a commutative diagram with
   exact rows
\[
\begin{tikzcd}
  0\arrow{r}&\coprod_{i\ge 0}\tilde X_i\arrow{r}\arrow{d}&
  \coprod_{i\ge 0}\tilde Y_i\arrow{r}\arrow{d}&
  \coprod_{i\ge 0}\tilde Z_i\arrow{r}\arrow{d}&0\\
  0\arrow{r}&\coprod_{i\ge 0}\tilde X_i\arrow{r}&
  \coprod_{i\ge 0}\tilde Y_i\arrow{r}&
  \coprod_{i\ge 0}\tilde Z_i\arrow{r}&0
\end{tikzcd}
\]
in $\A\tilde{\phantom{e}}$.  The vertical morphism are induced by the
morphisms $\tilde\eta_i\to \tilde\eta_{i+1}$, and taking mapping cones
of the vertical morphisms (given by cokernels in
$\A\tilde{\phantom{e}}$) yields the desired exact triangle
\eqref{eq:tria-colim}.

(2) The assumption on Cauchy sequences in $\C$ to be phantomless
implies that the functor $\widehat\C\to\T$ taking a sequence to its
homotopy colimit induces an equivalence
  \[\widehat\C\xto{\ _\sim\ }\D:=\{\hocolim_i X_i\in\T\mid X\in\Cau(\bbN,\C)\}.\]
  This follows from Lemma~\ref{le:phantomless}. In particular, the
  homotopy colimit of a Cauchy sequence in $\C$ is actually a
  colimit in $\D$.

  (3) We claim that $\D$ is a triangulated subcategory of $\T$ and
  that the exact triangles in $\D$ are up to isomorphism the colimits
  of Cauchy sequences given by coherent morphisms of exact triangles in $\C$.

  For a Cauchy sequence given by coherent morphisms of exact triangles
  $X_i\to Y_i\to Z_i\to\Si X_i$ in $\C$, we form in $\D$ its colimit
  and obtain an exact triangle \eqref{eq:tria-colim}; it does
  not depend on any choices.

  Conversely, fix an exact triangle
  $\eta\colon \bar X\to \bar Y\to \bar Z\to \Si \bar X$ in $\D$ that
  is given by $ X,Y\in\Cau(\bbN,\C)$ with $\bar X=\colim_i X_i$ and
  $\bar Y=\colim_i Y_i$. The morphism $\bar X\to \bar Y$ is up to
  isomorphism given by a morphism $\p\colon X\to Y$ in
  $\Cau(\bbN,\C)$, so of the form $\colim_i\p_i$.  Now complete the
  $\p_i\colon X_i\to Y_i$ to a sequence of coherent morphisms between
  exact triangles $X_i\to Y_i\to Z_i\to\Si X_i$ in $\C$. It is easily
  checked that $(Z_i)_{i\ge 0}$ is a Cauchy sequence; set
  $\tilde Z:=\colim_i Z_i$. This yields an exact triangle
  $\eta'\colon \bar X\to\bar Y\to\tilde Z\to\Si\bar X$ in $\T$,
  keeping in mind the above remark about homotopy colimits of exact
  triangles. It follows that $\D$ is closed under the formation of
  cones and therefore a triangulated subcategory of $\T$. Clearly,
  $\eta$ and $\eta'$ are isomorphic triangles. Thus any exact triangle
  in $\D$ is up to isomorphism a colimit of exact triangles in $\C$.
\end{proof}

\begin{cor}\label{co:tria-completion}
  Let $\T$ be an algebraic triangulated category with countable
  coproducts and $\C\subseteq\T$ a full triangulated subcategory
  consisting of compact objects. Suppose  the class of Cauchy
  sequences in $\C$ is phantomless. Then the full subcategory
  \[\{\hocolim_i X_i\in\T\mid X\in\Cau(\bbN,\C)\}\subseteq\T\] is a triangulated
  subcategory which is triangle equivalent to $\widehat\C$.  The exact
  triangles are precisely the ones isomorphic to colimits of Cauchy
  sequences given by coherent morphisms of exact triangles in
  $\C$.\qed
\end{cor}

For a generalisation of Corollary~\ref{co:tria-completion} from
algebraic triangulated categories to triangulated categories with a
morphic enhancement, see Section~\ref{ss:morphic-compact-completion}.

\begin{rem}\label{re:purity}
  To be phantomless is a condition which can be checked for a specific
  sequence $Y=(Y_i)_{i\ge 0}$ in $\C$.  Given $C\in\C$, call a
  subgroup
  \[U\subseteq\widetilde Y(C)=\colim_j\Hom(C,Y_j)\] \emph{of finite
    definition} if it arises as the image of a map
  $\widetilde Y(D)\to \widetilde Y(C)$ for some morphism $C\to D$ in $\C$;
  see \cite{GJ1981}. The descending chain condition (dcc) for subgroups of
  finite definition implies that \eqref{eq:ph-less} holds for all
  sequences $X$ in $\C$,
  since it implies the Mittag-Leffler condition for
  \[\cdots \lto \widetilde Y(X_2)\lto \widetilde Y(X_1) \lto \widetilde
    Y(X_0).\] The dcc for subgroups of finite definition  is equivalent
  to $Y$ being $\Si$-pure-injective when viewed as an object in
  $\Ind(\C)$, by \cite[\S3.5]{CB1994}. On the other hand, when $\T$ is a
  compactly generated triangulated category, then $Z\in\T$ is
  pure-injective if and only if $\Ph(-,Z)=0$, by
  \cite[Theorem~1.8]{Kr2000}.
\end{rem}

Let $\C$ be a triangulated category and fix a cohomological functor
$H\colon\C\to\A$ into an abelian category. Set $H^n:=H\circ\Si^n$ for
$n\in\bbZ$.  We call a sequence $(X_i)_{i\ge 0}$ in $\C$
\emph{bounded} if
\[\colim_iH^n(X_i)=0\quad\text{for} \quad |n|\gg 0\] and write
\[{\widehat\C}^b:=\{X\in\widehat\C\mid X \textrm{ bounded}\}\]
for the full subcategory of bounded objects.

\begin{rem}\label{re:bounded}  
  Suppose for all $C\in\C$ that $H^n(C)=0$ for $|n|\gg 0$.  Then we
  may restrict ourselves to bounded Cauchy sequences, and if this
  class is phantomless, then the conclusion of
  Theorem~\ref{th:tria-completion} holds for $\widehat\C^b$.

  More generally, fix a class $\X\subseteq\Fun(\bbN,\C)$ that is
  closed under suspensions and cones. When $\X$ is phantomless, then
  the conclusion of Corollary~\ref{co:tria-completion} holds for
  \[\widehat{\C_\X}\xto{\ _\sim\ }\{\hocolim_i X_i\in\T\mid X\in\X\}.\] For more details, cf.\
  Section~\ref{ss:completion}.
\end{rem}

\begin{exm}
  Let $\La$ be a quasi-Frobenius ring of finite representation type.
  Then the class of all sequences in the stable category
  $\St(\mod\La)$ is phantomless. In fact, this holds for any locally
  finite triangulated category \cite{Kr2012,XZ2005} and can be
  deduced from Remark~\ref{re:purity}.
\end{exm}

The following example is a continuation of our discussion in
Section~\ref{se:abelian}.  For an abelian category $\A$ let
$\bfD^b(\A)$ denote its bounded derived category.  An object in $\A$
is \emph{locally finite} if it is a directed union of finite length
subobjects.

\begin{exm}\label{ex:D(flA)}
  Let $k$ be a commutative ring and $\A$ a $k$-linear Grothendieck
  category such that $\Hom(X,Y)$ is a finite length $k$-module for all
  $X,Y\in\fl\A$.  Suppose that there are only finitely many
  isomorphism classes of simple objects and that the injective
  envelope of each simple object is locally finite and artinian. Then
  the class of Cauchy sequences in $\bfD^b(\fl\A)$ is phantomless and
  we have triangle equivalences
  \[\bfD^b(\widehat{\fl\A}) \xto{\
      _\sim\ }\bfD^b(\art\A)\xto{\ _\sim\ }\widehat{\bfD^b(\fl\A)}^b.\]
\end{exm}
\begin{proof}
  The first equivalence is clear from
  Proposition~\ref{pr:ind-artinian}; so we focus on the second one.
  We may assume that all objects in $\A$ are locally finite. For all
  $X,Y\in\bfD^b(\fl\A)$ the $k$-module $\Hom(X,Y)$ has finite length,
  since $\Ext^n(S,T)$ has finite length for all simple $S,T$ and
  $n \ge 0$ by our assumptions on $\A$.  It follows that the class of
  Cauchy sequences in $\bfD^b(\fl\A)$ is phantomless by the
  Mittag-Leffler condition.  We wish to apply
  Corollary~\ref{co:tria-completion} and choose for $\T$ the category
  $\bfK(\Inj\A)$ of complexes up to homotopy, where $\Inj\A$ denotes
  the full subcategory of injective objects in $\A$. Set
  $\inj\A=\Inj\A\cap\art\A$.  Then $\bfD^b(\fl\A)$ identifies with
  $\T^c$ via
  \[F\colon\bfD^b(\art\A)\iso \bfK^{+,b}(\inj\A)
    \hookrightarrow\bfK(\Inj\A),\]
  by \cite[Proposition~2.3]{Kr2005}. Set $\C:=\bfD^b(\fl\A)$.  Then
  Corollary~\ref{co:tria-completion} yields a triangle equivalence
 \[\widehat\C\xto{\ _\sim\ }\D:=\{\hocolim_i X_i\in\T\mid X\in\Cau(\bbN,\C)\}.\]
 We claim
 that $F$ induces an equivalence
 \[\bfD^b(\art\A)\iso\D^b:=\{X\in\D\mid \colim_iH^n(X_i)=0\text{ for } |n|\gg 0\}.\]
 Any object $M\in\art\A$ is the colimit of the Cauchy sequence
 $(\soc^iM)_{i\ge 0}$ in $\fl\A$ by Proposition~\ref{pr:ind-artinian},
 and this yields a Cauchy sequence in $\bfD^b(\fl\A)$. Thus $F$ maps
 $\art\A$ into $\D^b$, and therefore also $\bfD^b(\art\A)$, since
 $F$ is an exact functor and $\bfD^b(\art\A)$ is generated by $\art\A$
 as a triangulated category. Conversely, let $\bar X=\colim_i X_i$ be
 an object in $\D^b$. We may assume that the complex $\bar X$ is
 homotopically minimal, as in \cite[Appendix~B]{Kr2005}. The Cauchy
 condition implies for each simple $S\in\A$ and $n\in\bbZ$ that $\Hom(S,\Si^n \bar X)$ has
 finite length over $k$, so the degree $n$ component of $\bar X$ is
 artinian. Thus $\bar X$ belongs to $ \bfK^{+,b}(\inj\A)$, and this
 yields the claim.
\end{proof}

\section{Homologically perfect objects}

Let $\T$ be a compactly generated triangulated category and denote by
$\T^c$ the full subcategory of compact objects. We fix a cohomological
functor $H\colon\T\to\A$ into an abelian category.  For $n\in\bbZ$
set $H^n:=H\circ\Si^n$.

\begin{defn}
  We say that an object $X$ in $\T$ is \emph{homologically perfect}
  (with respect to $H$) if $X$ can be written as homotopy colimit of a
  sequence $X_0\to X_1\to X_2\to\cdots$ in $\T^c$ such that the
  following holds:
  \begin{enumerate}
  \item[(HP1)] The sequence $(X_i)_{i\ge 0}$ is a Cauchy sequence in
    $\T^c$, that is, for every $C\in \T^c$
\[\Hom(C,X_i)\xto{_\sim}\Hom(C,X_{i+1})\quad\textrm{for}\quad i\gg 0.\]
\item[(HP2)] For every $n\in\bbZ$ we have $H^n(X_i)\iso H^n(X_{i+1})$
  for $i\gg 0$.
\item[(HP3)] For almost all $n\in\bbZ$ we have $H^n(X_i)=0$ for
  $i\gg 0$.
\end{enumerate}
When $\T$ is the derived category of an abelian category, then our
choice of $H$ is the natural one given by the degree zero cohomology
of a complex, unless stated otherwise.
\end{defn}

It is clear that the above definition depends on the choice of $H$; so
a different choice of $H$ may yield a different class of homologically
perfect objects. However, in our applications there are natural
choices for $H$, for instance when $\T$ is the derived category of an
abelian category. It is a remarkable fact that in those cases
there is an intrinsic description of homologically perfect objects
that depends only on $\T$. That means some choices of $H$ are more
natural than others.  The following Lemma~\ref{le:intrinsic} makes
this precise when the cohomological functor $H$ is given by a compact
generator. For noetherian schemes that are non-affine, the
 natural choice for $H$ admits the same intrinsic
description of homologically perfect objects, but the proof is more
involved and we refer to Theorem~\ref{thm:intrinsic}.

We begin with an elementary but useful observation.

\begin{lem}\label{le:cauchy-generator}
  Let $\C$ be a triangulated category and $G\in\C$ an object that
  generates $\C$, that is, $\C$ admits no proper thick subcategory
  containing $G$. Then a sequence $X_0\to X_1\to X_2\to\cdots$ in $\C$
  is Cauchy if and only if for all $n\in\bbZ$ we have
  $\Hom(\Si^n G,X_i)\xto{_\sim}\Hom(\Si^n G,X_{i+1})$ for
  $i\gg 0$.\qed
\end{lem}

The following yields  an intrinsic description of homologically perfect objects.
  
\begin{lem}\label{le:intrinsic}
  Let $G$ be a compact object in $\T$ that generates $\T^c$ as a
  triangulated category. Then for $X\in\T$ the following are equivalent:
 \begin{enumerate}
  \item The object $X$ is homologically perfect with
    respect to $H=\Hom(G,-)$.
    \item The object $X$ can be written as
  homotopy colimit of a Cauchy sequence in
  $\T^c$, and for every $C\in \T^c$ we have $\Hom(C,\Si^nX)=0$ for $|n|\gg 0$.
\end{enumerate}
\end{lem}
\begin{proof}
  The assumption on $H$ implies that the conditions (HP1) and (HP2)
  are equivalent, thanks to Lemma~\ref{le:cauchy-generator}. Condition
  (HP3) means $H^n(X)=0$ for $|n|\gg 0$, since  $H^n(X)\cong\colim_iH^n(X_i)$
  by Lemma~\ref{le:hocolim}. Thus (HP3) is equivalent to the
  condition that for every $C\in \T^c$ we have $\Hom(C,\Si^nX)=0$ for
  $|n|\gg 0$, since $G$ generates $\T^c$.
\end{proof}

Now fix a ring $\La$. We write $\bfD(\La)$ for the derived category of
the abelian category of all $\La$-modules. Let $\bfD^\per(\La)$ denote
the full subcategory of \emph{perfect complexes}, that is, objects
isomorphic to bounded complexes of finitely generated projective
modules. The triangulated category $\bfD(\La)$ is compactly generated
and the compact objects are precisely the perfect complexes.

Let $\mod\La$ denote the category of finitely presented $\La$-modules
and $\proj\La$ denotes the full subcategory of finitely generated
projective modules.  When $\La$ is a right coherent ring,
then $\mod\La$ is abelian and we consider its derived category
$\bfD^b(\mod\La)$ using the following identifications
\[\begin{tikzcd}
    \bfK^b(\proj\La)\arrow[r,hook]\arrow{d}{\wr}& \bfK^{-,b}(\proj\La)\arrow{d}{\wr}\\
    \bfD^\per(\La)\arrow[r,hook]&\bfD^b(\mod\La)
  \end{tikzcd}\] where the top row consists of categories of complexes
of modules in $\proj\La$ up to homotopy.  Note that
$\bfD^\per(\La)=\bfD^b(\mod\La)$ if and only if every finitely
presented $\La$-module has finite projective dimension.

We provide an intrinsic description of the objects from
$\bfD^b(\mod\La)$,  which  uses for any complex $X$ the sequence of truncations
\[\cdots\lto \s_{\ge n+1}X\lto \s_{\ge n}X\lto \s_{\ge n-1}X\lto\cdots\]
given by
\[
\begin{tikzcd}
\s_{\ge n}X\arrow{d}& \cdots \arrow{r}&
  0\arrow{r}{}\arrow{d}&0\arrow{r}{}\arrow{d}&
X^{n}\arrow{r}\arrow{d}{\id}&X^{n+1}\arrow{r}\arrow{d}{\id}&\cdots \\ 
X&\cdots\arrow{r}&
X^{n-2}\arrow{r}{}&X^{n-1}\arrow{r}{}&X^n\arrow{r}&X^{n+1}\arrow{r}&\cdots .
\end{tikzcd}
\]

In the following lemma we use the notion of a homologically perfect
object with respect to the functor that takes degree zero cohomology of a
complex,  keeping in mind Lemma~\ref{le:intrinsic}.

\begin{lem}\label{le:coherent-perf-cohom}
  Let $\La$ be a right coherent ring. Then $X$ in $\bfD(\La)$ is
  homologically perfect if and only if $X$ belongs to
  $\bfD^b(\mod\La)$.
\end{lem}
\begin{proof}
  Let $X$ be a complex in $\bfK^{-,b}(\proj\La)=\bfD^b(\mod\La)$ and
  write $X$ as homotopy colimit of its truncations $X_i=\s_{\ge -i}X$
  which lie in $\bfK^{b}(\proj\La)$. It is clear that $X$ is
  homologically perfect. In fact, $\bfD^\per(\La)$ is generated by
  $\La$; so it suffices to check the functor $H^n=\Hom(\La,\Si^{n}-)$
  for every $n\in\bbZ$. We have $H^n(X_i)\xto{_\sim}H^n(X_{i+1})$ for
  $i\gg 0$ and $H^nX=0$ for  $|n|\gg 0$.  On the other hand, if $X$
  is homologically perfect, then $H^nX$ is finitely presented for all
  $n$, so $X$ lies in $\bfD^b(\mod\La)$.
\end{proof}

\section{The bounded derived category}

Let $\La$ be a ring. We consider the category $\mod\La$ of finitely
presented $\La$-modules and its bounded derived category
$\bfD^b(\mod\La)$.  Our aim is to identify $\bfD^b(\mod\La)$ with a
completion of $\bfD^\per(\La)$ when $\La$ is right coherent; compare
this with Rouquier's \cite[Corollary~6.4]{Ro2008}.

\begin{lem}\label{le:bounded}
  Let $\La$ be a ring and set $\P=\proj\La$.  Then the functor
  \[\bfK^{-,b}(\P)\lto \Add(\bfK^b(\P)^\op,\Ab),\quad
    X\mapsto h_X:=\Hom(-,X)|_{\bfK^b(\P)},\] is fully faithful.
\end{lem}
\begin{proof}
  We view $\bfK^{-,b}(\P)$ as a subcategory of $\bfD(\La)$.  Let $X,Y$
  be objects in $\bfK^{-,b}(\P)$ and write $X$ as homotopy colimit of
  its truncations $X_i=\s_{\ge -i}X$ which lie in $\bfK^{b}(\P)$.  Let
  $C_i$ denote the cone of $X_{i}\to X_{i+1}$. This complex is
  concentrated in degree $-i-1$; so $\Hom(C_i,Y)=0$ for $i\gg 0$. Thus
  $X_{i}\to X_{i+1}$ induces a bijection
\[\Hom(X_{i+1},Y)\xto{_\sim}\Hom(X_{i},Y) \quad\textrm{for}\quad i\gg
  0.\] This implies
 \[\Hom(X,Y)\xto{_\sim}\lim_i\Hom(X_i,Y)\]
 and therefore $\Ph(X,Y)=0$ by Lemma~\ref{le:phantom}.  From
 Lemma~\ref{le:yoneda} we conclude that
  \[\Hom(X,Y)\xto{_\sim}\Hom(h_X,h_Y).\qedhere\]
\end{proof}

Let $\C$ be a triangulated category and fix a cohomological functor
$H\colon\C\to\A$.  Recall that an object $X$ in $\widehat\C$ is
\emph{bounded} if $\colim_iH^n(X_i)=0$ for $|n|\gg 0$, and
$\widehat\C^b$ denotes the full subcategory of bounded objects in
$\widehat\C$. From Theorem~\ref{th:tria-completion} and
Remark~\ref{re:bounded} we know that $\widehat\C^b$ admits a canonical
triangulated structure when $\C$ is algebraic and bounded Cauchy
sequences are phantomless.

\begin{thm}\label{th:bounded-derived}
For a right coherent ring $\La$ there is a canonical triangle equivalence
\[\widehat{\bfD^\per(\La)}^b\xto{\ _\sim\ }\bfD^b(\mod\La)\]
which sends a Cauchy sequence in $\bfD^\per(\La)$ to its colimit.
\end{thm}
\begin{proof}
  We consider the functor
  \[\bfD^b(\mod\La)\lto\Add(\bfD^\per(\La)^\op,\Ab),\quad
    X\mapsto\Hom(-,X)|_{\bfD^\per(\La)},\] which is fully faithful by
  Lemma~\ref{le:bounded}. On the other hand, we have the functor
  \[\widehat{\bfD^\per(\La)}^b\lto\Add(\bfD^\per(\La)^\op,\Ab), \quad
    X\mapsto\widetilde X,\] which is fully faithful by
  Proposition~\ref{pr:completion}.  Both functors have the same
  essential image by Lemma~\ref{le:coherent-perf-cohom}, because we
  can identify this with a full subcategory of $\bfD(\La)$ by
  Lemma~\ref{le:phantomless}. This yields a triangle equivalence,
  since the triangulated structures of both categories identify with
  the one from $\bfD(\La)$; see Corollary~\ref{co:tria-completion}
  plus Remark~\ref{re:bounded}.
\end{proof}

\begin{rem}
The triangulated category $\bfD^\per(\La)$ admits a morphic enhancement
which is given by $\bfD^\per(\La_1)$, with $\La_1$ the ring of upper
triangular $2\times2$ matrices over $\La$. This enhancement can be
completed and yields a morphic enhancement of
$\widehat{\bfD^\per(\La)}^b$ that identifies with the morphic
enhancement of $\bfD^b(\mod\La)$; see
Section~\ref{ss:completion-morphic-enrich}. This observation enriches
the triangle equivalence of Theorem~\ref{th:bounded-derived}.
\end{rem}

For a noetherian algebra over a complete local ring, there is another
description of $\bfD^b(\mod\La)$ which is obtained by completing the
category of finite length modules over $\La^\op$. 

\begin{prop}
  Let $\La$ be a noetherian algebra over a complete local ring and set
  $\Ga=\La^\op$.  Then
  there are triangle equivalences
  \[\bfD^b(\mod\La)^\op\xto{\ _\sim\ }\bfD^b(\widehat{\fl\Ga})
    \xto{\ _\sim\ }\widehat{\bfD^b(\fl\Ga)}^b.\]
\end{prop}
\begin{proof}
  Matlis duality gives an equivalence $(\mod\La)^\op\iso \art\Ga$, so
  $\bfD^b(\mod\La)^\op\iso \bfD^b(\art\Ga)$, and we have
  $\art\Ga\iso \widehat{\fl\Ga}$ by Example~\ref{ex:noeth-alg}.  This
  yields the first functor, and the second is from
  Example~\ref{ex:D(flA)}.
\end{proof}

\section{Pseudo-coherent objects}\label{se:pseudo-coherence}

Let $\T$ be a triangulated category and $H\colon\T\to\A$ a
cohomological functor into an abelian category.  Set
\[\T^{> n}:=\{X\in\T\mid H^iX=0\textrm{ for all }i \le n\}\]
and
\[\T^{\le n}:=\{X\in\T\mid H^iX=0\textrm{ for all }i > n\}.\]
We suppose  for all $X,Y\in\T$ and $n\in\bbZ$ the following:
\begin{enumerate}
  \item[(TS1)] There is an exact triangle
\[\t_{\le n}X\lto X\lto\t_{>n}X \lto \Si (\t_{\le n}X)\]
with $\t_{\le n}X\in \T^{\le n}$ and $\t_{> n}X\in \T^{> n}$.
\item[(TS2)] $\Hom(X,Y)=0$ for  $X\in \T^{\le n}$ and $Y\in \T^{>n}$.
\end{enumerate}
Thus the category $\T$ is equiped with a \emph{t-structure}
\cite{BBD1982}.

We will use the following observation.

\begin{lem}
 For any morphism $X\to Y$ in $\T$ we have
\[\t_{>n}X\xto{_\sim}\t_{>n}Y\quad\iff\quad H^iX\xto{_\sim}H^iY\textrm{
    for all } i>n.\]
\end{lem}

\begin{proof}
Note that for any object $X$ the morphism
$X\to\t_{>n}X$ induces an isomorphism $H^iX\to H^i(\t_{>n}X)$ for all
$i>n$. Thus $ H^iX\xto{_\sim}H^iY$ for all $i>n$ if and only if
$H^i(\t_{>n}X)\xto{_\sim}H^i(\t_{>n}Y)$ for all $i\in\bbZ$.
\end{proof}

Now suppose that $\T$ is compactly generated and write $\T^c$ for the
full subcategory of compact objects.

\begin{defn}
  An object $X\in\T$ is called \emph{pseudo-coherent} (with respect to
  the chosen t-structure) if $X$ can be written as homotopy colimit of
  a sequence $X_0\to X_1\to X_2\to\cdots$ in $\T^c$ such that
  $\t_{>-i} X_i\xto{_\sim} \t_{> -i} X$ for all $i\ge 0$. We say that
  $X$ has \emph{bounded cohomology} if $H^nX=0$ for $|n|\gg 0$.
\end{defn}

\begin{lem}\label{le:pseudo-coh}
  The  functor
  \[\T\lto\Add((\T^c)^\op,\Ab),\quad X\mapsto h_X:=\Hom(-,X)|_{\T^c},\]
  is fully faithful when restricted to pseudo-coherent objects with
  bounded cohomology.
\end{lem}

\begin{proof}
  Let $X,Y$ be objects in $\T$. Suppose that $X= \hocolim_iX_i$ is
  pseudo-coherent and $H^nY=0$ for $n\ll 0$.  Let $C_i$ denote the
  cone of $X_{i}\to X_{i+1}$.  The induced morphism
  $\t_{> -i}X_i\to\t_{>-i}X_{i+1}$ is an isomorphism since
  $\t_{>-i}\t_{>-(i+1)}=\t_{>-i}$. Thus $C_i\in\T^{\le -i}$ and
  therefore $\Hom(X_{i+1},Y)\xto{_\sim}\Hom(X_{i},Y)$ for $i\gg 0$. It
  follows from  Lemma~\ref{le:phantom} that $\Ph(X,Y)=0$, so 
  \[\Hom(X,Y)\xto{_\sim}\Hom(h_X,h_Y)\] by Lemma~\ref{le:yoneda}.
\end{proof}

\begin{exm}
  Let $\La$ be a ring and  $\T=\bfD(\La)$ the derived category
  of the category of all $\La$-modules with the standard
  t-structure. Then the canonical functor
  $\bfK^-(\proj\La)\to\bfD(\La)$ identifies $\bfK^-(\proj\La)$ with
  the full subcategory of pseudo-coherent objects in $\bfD(\La)$.
\end{exm}
\begin{proof}
  For $X\in\bfK^-(\proj\La)$ and $i\ge 0$ set $X_i:=\s_{\ge -i}X$. Then
  we have $X=\hocolim_i X_i=X$ and
  $\t_{>-i} X_i\xto{_\sim} \t_{> -i} X$ for all $i\ge 0$. Thus $X$ is
  pseudo-coherent. The other implication is left to the reader.
\end{proof}

The example shows that for a right coherent ring $\La$ and any object
$X$ in $\T=\bfD(\La)$ the following are equivalent:
\begin{enumerate}
\item[(PC)] $X$ is pseudo-coherent and has bounded cohomology.
\item[(HP)] $X$ is homologically perfect.
\end{enumerate}
This seems to be a common phenomenon (cf.\
Propositions~\ref{pr:coherent-cohom} and \ref{prop:ringspectra})
though we do not have a general proof.

Let $\C$ be a triangulated category and fix a cohomological functor
$H\colon\C\to\A$.  Call a sequence $(X_i)_{i\ge 0}$ in $\C$
\emph{strongly bounded} if $\colim_iH^n(X_i)=0$ for $|n|\gg 0$, and if
for every $n\in\bbZ$ we have $H^n(X_i)\iso H^n(X_{i+1})$ for $i\gg
0$. By abuse of notation, we write $\widehat\C^b$ for the full
subcategory of strongly bounded objects in $\widehat\C$.\footnote{The
  condition $H^n(X_i)\iso H^n(X_{i+1})$ for $i\gg 0$ is automatic for
  a Cauchy sequence $X$ when $H=\Hom(C,-)$ for an object
  $C\in\C$.}

\begin{lem}\label{re:pseudo-coherent}
  Suppose that \emph{(PC)} $\Leftrightarrow$ \emph{(HP)} for all $X\in\T$, and set
  $\C:=\T^c$. Then the functor
  \[F\colon\widehat{\C}^b\lto\T,\quad X\mapsto\hocolim_i X_i,\] is
  fully faithful functor and identifies $\widehat{\C}^b$ with the full
  subcategory of pseudo-coherent objects having bounded
  cohomology. When $\T$ admits a morphic enhancement, then $F$
  is a triangle functor.
\end{lem}
\begin{proof}
  The first assertion follows from Lemmas~\ref{le:phantomless} and
  \ref{le:pseudo-coh}. For the second assertion, see
  Lemma~\ref{le:morphic-compact-completion}.
\end{proof}

\section{Noetherian schemes}

We fix a noetherian scheme $\bbX$. Let $\Qcoh\bbX$ denote the category
of quasi-coherent sheaves on $\bbX$, and $\coh\bbX$ denotes the full
subcategory of coherent sheaves. We consider the derived categories
\[\bfD^\per(\bbX)\hookrightarrow 
  \bfD^b(\coh \bbX)\hookrightarrow \bfD(\Qcoh\bbX).\] The triangulated
category $\bfD(\Qcoh\bbX)$ is compactly generated and the full
subcategory of compact objects agrees with the category
$\bfD^\per(\bbX)$ of perfect complexes \cite{Ne1996}.  We use the
standard t-structure and then the above notion of a pseudo-coherent
object identifies with the usual one; see \cite[\S2.3]{Il1966},
\cite[\S2.2]{TT1990}, and
\cite[\href{http://stacks.math.columbia.edu/tag/0DJM}{\S0DJM}]{stacks-project}.
A precise reference is
\cite[\href{http://stacks.math.columbia.edu/tag/0DJN}{Lemma~0DJN}]{stacks-project},
which uses approximations and builds on work of Lipman and Neeman
\cite{LN2007}.

We obtain the following description of the objects in
$\bfD^b(\coh\bbX)$. For a refinement, see Theorem~\ref{thm:intrinsic}.
We use the notion of a homologically perfect object with respect to
the functor that takes degree zero cohomology of a complex.

\begin{prop}\label{pr:coherent-cohom}
  For an object $X$ in $\bfD(\Qcoh\bbX)$ the following are equivalent:
  \begin{enumerate}
  \item  $X$ belongs to $\bfD^b(\coh\bbX)$.
  \item $X$ is pseudo-coherent and has bounded cohomology.
  \item $X$ is homologically perfect.
\end{enumerate}
\end{prop}
\begin{proof}
  (1) $\Leftrightarrow$ (2): See \cite[Example~2.2.8]{TT1990}.

  (2) $\Rightarrow$ (3): Let $X=\hocolim_i X_i$ be pseudo-coherent and
  $C$ a perfect complex. The argument in the proof of
  Lemma~\ref{le:pseudo-coh} shows that $(X_i)_{i\ge 0}$ is a Cauchy
  sequence in $\bfD^\per(\bbX)$. More precisely, the cone of
  $X_{i}\to X_{i+1}$ belongs to $\T^{\le -i}$, and therefore
  $\Hom(C,X_i)\xto{_\sim}\Hom(C,X_{i+1})$ for $i\gg 0$; see
  \cite[\href{http://stacks.math.columbia.edu/tag/09M2}{\S09M2}]{stacks-project}.
  Also, $H^n(X_i)\iso H^n(X_{i+1})$ for all $i>-n$.  Finally, for
  almost all $n\in\bbZ$ we have $H^n(X_i)\cong H^n(X)=0$ for $i\gg 0$,
  since $X$ has bounded cohomology.

  (3) $\Rightarrow$ (1): Let $X=\hocolim_i X_i$ be homologically
  perfect and $n\in\bbZ$.  Then $H^n(X)\cong \colim_i H^n(X_i)$ equals
  the cohomology of some perfect complex, so $H^n(X)$ is coherent. Also,
  $H^n(X)=0$ for $|n|\gg 0$.
\end{proof}

The following is now the analogue of Theorem~\ref{th:bounded-derived}
for schemes that are not necessarily affine. The proof is very
similar; see also Lemma~\ref{re:pseudo-coherent} for the general
argument.

\begin{thm}\label{th:scheme-bounded-derived}
  For a  noetherian scheme $\bbX$ there is a canonical
  triangle equivalence
\[\widehat{\bfD^\per(\bbX)}^b\xto{\ _\sim\ }\bfD^b(\coh\bbX)\]
which sends a Cauchy sequence in $\bfD^\per(\bbX)$ to its colimit.\end{thm}
\begin{proof}
 We consider the functor
  \[\bfD^b(\coh\bbX)\lto\Add(\bfD^\per(\bbX)^\op,\Ab),\quad
    X\mapsto\Hom(-,X)|_{\bfD^\per(\bbX)},\] which is fully faithful by
  Lemma~\ref{le:pseudo-coh} and Proposition~\ref{pr:coherent-cohom}.
  On the other hand, we have the functor
  \[\widehat{\bfD^\per(\bbX)}^b\lto\Add(\bfD^\per(\bbX)^\op,\Ab), \quad
    X\mapsto\widetilde X,\] which is fully faithful by
  Proposition~\ref{pr:completion}. Both functors have the same
  essential image by Proposition~\ref{pr:coherent-cohom}, because we
  can identify this with a full subcategory of $\bfD(\Qcoh\bbX)$ by
  Lemma~\ref{le:phantomless}. This yields a triangle equivalence,
  since the triangulated structures of both categories identify with
  the one from $\bfD(\Qcoh\bbX)$; see
  Corollary~\ref{co:tria-completion} plus Remark~\ref{re:bounded}.
\end{proof}

An immediate consequence is the following.

\begin{cor}
  \pushQED{\qed}
  The singularity category of $\bbX$
(in the sense of Buchweitz and Orlov \cite{Bu1987,Or2004}) identifies
with the Verdier quotient
\[\frac{\widehat{\bfD^\per(\bbX)}^b}{\bfD^\per(\bbX)}.\qedhere\]
\end{cor}

\appendix

\section{Homologically perfect objects in homotopy theory \\[.5em]
  by Tobias Barthel}
  
Let $R$ be an associative ring spectrum and let $\D_R$ be the derived
category of right $R$-module spectra as constructed for example in
\cite{EKMM1997}; if no confusion is likely to arise, we will refer to
an object in $\D_R$ simply as an $R$-module. If $R$ is connective,
then $\D_R$ inherits the standard t-structure from the stable homotopy
category, and we denote by $\D_R^b\subseteq\D_R$ the full triangulated
subcategory of bounded $R$-module spectra, i.e., those $R$-modules $M$
with $\pi_iM$ finitely presented over $\pi_0R$ and $\pi_iM = 0$ for
$\lvert i \rvert \gg 0$. As usual, $\D_R^c \subseteq \D_R$ is the full
subcategory of compact $R$-modules or, equivalently, the thick
subcategory generated by the right $R$-module $R$.

Throughout this appendix, we will employ homological grading, so for
example the pseudo-coherence condition
$\tau_{>-i}M_i \xrightarrow{_\sim} \tau_{>-i}M$ introduced in
Section~\ref{se:pseudo-coherence} translates to
$\pi_jM_i \xrightarrow{_\sim} \pi_jM$ for all $j <i$. Furthermore, the
notion of homologically perfect object used here will always be with
respect to the homological functor $\pi_0$.

\begin{prop}\label{prop:ringspectra}
  Suppose $R$ is a connective associative ring spectrum with $\pi_0R$
  right coherent and $\pi_iR$ finitely presented over $\pi_0R$ for all
  $i \ge 0$. The following conditions on $M \in \D_R$ are equivalent:
	\begin{enumerate}
		\item $M$ is pseudo-coherent and has bounded homotopy.
		\item $M$ is homologically perfect.
		\item $M$ belongs to $\D_R^b$.
	\end{enumerate}
\end{prop}
\begin{proof}
  The implication (1) $\Rightarrow$ (2) is proven as in
  Proposition~\ref{pr:coherent-cohom}[(2) $\Rightarrow$ (3)]: indeed,
  it suffices to test against compact objects of the form
  $C = \Sigma^nR$ for all $n\in \bbZ$, for which the claim is
  clear. In order to see that (2) $\Rightarrow$ (3), we first observe
  that any compact $R$-module has finitely presented homotopy groups
  by assumption on $R$. Therefore, any homologically perfect
  $M \in \D_R$ can only have finitely many nonzero homotopy groups,
  all of which must be finitely presented over $\pi_0R$ by the Cauchy
  condition. Thus, $M$ is bounded.

  Now assume that $M \in \D_R^b$, then $M$ has bounded homotopy and
  it remains to show that $M$ has to be pseudo-coherent. To this end,
  we use a mild variant of the cellular tower construction of
  \cite[Thm.~III.2.10]{EKMM1997} or \cite[Prop.~2.3.1]{HPS1997}, in
  which we only attach $R$-cells of a fixed dimension in each
  step. Indeed, let $M \in \D_R$ be a bounded below $R$-module with
  finitely presented homotopy groups and assume without loss of
  generality that the lowest nonzero homotopy group is in degree
  $0$. Inductively, we construct a tower of $R$-modules
  $(M^k)_{k \ge 0}$ under $M$ with:
\begin{enumerate}
\item[(i)] $M^0 = M$.
\item[(ii)] For all $k\ge 0$, there is a cofiber sequence 
  \[
    F^k=\bigoplus_{G(k)} \Sigma^kR \lto M^k \lto M^{k+1},
\]
where the direct sum is indexed by a set $G(k)$ of generators of the finitely
presented $\pi_0R$-module $\pi_kM^k$.
\end{enumerate}

It follows by induction on $k$, the assumption on $R$, and the long
exact sequence in homotopy that $\pi_*M^{k+1}$ is finitely presented
over $\pi_0R$ in all degrees and zero below degree $k+1$, which allows
us to construct the map $F^{k+1} \to M^{k+1}$ and to proceed with the
induction.
 
Set $M_k=\fib(M \to M^k)$. The octahedral axiom provides fiber sequences 
\begin{equation}\label{eq:fibersequence}
M_k \lto M_{k+1} \lto F^k 
\end{equation}
and a sequence of $R$-modules
\[
M_0 \lto M_1 \lto M_2 \lto \ldots
\]
over $M$. The fiber sequences \eqref{eq:fibersequence} imply that
$\pi_iM_k \xrightarrow{_\sim} \pi_iM_{k+1}$ and hence
$\pi_iM_k \xrightarrow{_\sim} \pi_iM$ for all $i<k$, which then also
shows that the homotopy colimit over $(M_k)_{k\ge 0}$ is equivalent to
$M$ by a connectivity argument. Moreover, because $M_k$ is built from
finitely many $R$-cells, $M_k$ is compact for any $k\ge 0$. It follows
that $M$ is pseudo-coherent as desired.
\end{proof}

In light of Lemma~\ref{re:pseudo-coherent}, we obtain the following
consequence.

\begin{cor}
  With notation as in Proposition~\ref{prop:ringspectra}, taking homotopy
  colimits induces an equivalence
  $\widehat{\D_R^c}^b \xrightarrow{_\sim}
  \D_R^b$ of triangulated categories.\qed
\end{cor}

In particular, the corollary applied to the Eilenberg--Mac~Lane ring
spectrum $H\Lambda$ of a right coherent ring $\Lambda$ recovers
Theorem~\ref{th:bounded-derived}.

\begin{lem}
  Let $R$ be as in Proposition~\ref{prop:ringspectra} and assume
  additionally that the right global dimension of $\pi_0R$ is finite,
  then $\D_R^b$ coincides with the thick subcategory
  of $\D_R^b$ generated by the Eilenberg--Mac~Lane $R$-module
  $H\pi_0R$.
\end{lem}
\begin{proof}
  Since $\pi_0R$ has finite right global dimension, any finitely
  presented $\pi_0R$-module $N$ admits a finite length resolution by
  finitely presented projective $\pi_0R$-modules. This implies that the
  $R$-module spectrum $HN$ belongs to $\Thick(H\pi_0R)$. A
  Postnikov tower argument then shows that any $M \in \D_R^b$
  is in $\Thick(H\pi_0R)$.

  Conversely, since the $R$-module $H\pi_0R$ is bounded, so is any $R$-module
   that belongs to $\Thick(H\pi_0R)$.
\end{proof}

The stable homotopy category identifies with $\D_R$ for the sphere
spectrum $R = S^0$ and we obtain the following consequence, a variant
of which has appeared independently in work of
Neeman~\cite[Ex.~22]{Ne2019}.

\begin{cor}
  For a spectrum $X$ in the stable homotopy category the following
  conditions are equivalent:
\begin{enumerate}
\item $X$ is pseudo-coherent and has bounded homotopy.
\item $X$ is homologically perfect.
\item $X$ belongs to the thick subcategory generated by the spectrum
  $H\bbZ$.
\end{enumerate}
In particular, a homologically perfect spectrum $X$ is compact if and
only if $X = 0$.
\end{cor}

 \begin{proof}
   By the finite generation of the stable homotopy groups of spheres
   and because $\pi_0S^0 \cong \bbZ$, the sphere spectrum $R = S^0$
   satisfies the conditions of the previous lemma, so we have
   $\D_{S^0}^b=\Thick(H\bbZ)$.
  
  A theorem of Serre says that a finite spectrum must have infinitely
  many nonzero homotopy groups, so this result implies in particular
  that a homologically perfect spectrum $X$ is compact if and only if
  $X = 0$.
\end{proof}

Note that, in contrast to the perfect derived categories of right
coherent rings or noetherian schemes, the category $\D_{S^0}^b$ does
not contain $\D_{S^0}^c$ as a subcategory.

\section{Homologically perfect objects on noetherian schemes\\[.5em]
by Tobias Barthel and Henning Krause}

Throughout this appendix, all schemes will be assumed to be separated
and noetherian. For a scheme $\bbX$, we write $\bfD(\Qcoh\bbX)$ for
the derived category of quasi-coherent sheaves on $\bbX$ and $G_\bbX$
denotes a compact generator of $\bfD(\Qcoh\bbX)$, which exists by
\cite[Thm.~3.3.1]{BVB2003}. The notion of homologically perfect
object considered in this appendix is defined with respect to the
cohomological functor $H^0$, taking the degree zero cohomology of an
object in $\bfD(\Qcoh\bbX)$. Our goal is to give an intrinsic
description of the homologically perfect objects in $\bfD(\Qcoh\bbX)$
that does not depend on the chosen t-structure.  For affine schemes,
this has already been observed in Lemma~\ref{le:intrinsic}, but the non-affine
case is more complicated and relies crucially on Neeman's work on
approximability \cite{Ne2017,Ne2018}. The main result is:

\begin{thm}\label{thm:intrinsic} 
  Let $\bbX$ be a separated noetherian scheme. An object
  $X\in\bfD(\Qcoh\bbX)$ belongs to $\bfD^b(\coh\bbX)$ if and only if
  $X$ is the homotopy colimit of a sequence
  $X_0 \to X_1 \to X_2 \to \cdots$ of perfect complexes on $\bbX$
  satisfying the following conditions for every
  $C \in \bfD^\per(\bbX)$:
\begin{enumerate}
\item $\Hom(C,X_s) \iso \Hom(C,X_{s+1})$ for $s \gg
  0$, and
\item $\Hom(C,\Sigma^nX) = 0$ for $\lvert n \rvert \gg 0$.
\end{enumerate}
\end{thm}

Before we give the proof of the theorem at the end of this appendix,
we record the following consequence, which is an immediate application
of Proposition~\ref{pr:coherent-cohom} and
Theorem~\ref{thm:intrinsic}. It shows in particular that the
completion
\[\widehat{\bfD^\per(\bbX)}^b\xto{\ _\sim\ }\bfD^b(\coh\bbX)\] depends
only on the triangulated structure of $\bfD^\per(\bbX)$.

\begin{cor} The homologically perfect objects on a separated
noetherian scheme $\bbX$ with respect to the standard t-structure
depend only on the triangulated structure of $\bfD(\Qcoh\bbX)$.\qed
\end{cor}

We will prepare for the proof of the theorem with three lemmata which
make use of Neeman's study of strong generators and approximability
for triangulated categories.

\begin{lem}\label{lem:pullback}
  Let $i\colon \bbU \to \bbX$ be an open immersion.  If $(X_s)_s$ is a
  Cauchy sequence in $\bfD^\per(\bbX)$, then $(i^*X_s)_s$ is a Cauchy
  sequence in $\bfD^\per(\bbU)$.
\end{lem}
\begin{proof} First note that, because $i$ is an open immersion and
thus automatically quasi-compact, $i^*$ exhibits $\bfD(\Qcoh \bbU)$ as
the essential image of a smashing Bousfield localisation on
$\bfD(\Qcoh\bbX)$. Therefore, $i^*$ preserves homotopy colimits and
compact object and it has a fully faithful right adjoint $\bfR{i}_*$. In
particular, $i^*G_\bbX$ is a compact generator of $\bfD(\Qcoh \bbU)$,
which we will denote by $G_\bbU$. Moreover, it follows from the
projection formula that for every $Y \in \bfD(\Qcoh\bbX)$ there is a
canonical quasi-isomorphism
\begin{equation}\label{eq:smashing} \bfR{i}_*i^*Y \simeq (\bfR{i}_*\cO_\bbU)
\Lotimes Y,
\end{equation} where $\cO_\bbU$ is the structure sheaf of $\bbU$.

In order to prove the lemma, it suffices by Lemma~\ref{le:cauchy-generator}
to show that for every $k \in \bbZ$ there exists an $s(k)$ such that
for all $s>s(k)$:
\[ \Hom(\Sigma^kG_\bbU, i^*X_s) \iso
\Hom(\Sigma^kG_\bbU, i^*X_{s+1}).
\] Without loss of generality, we will demonstrate the existence of
$s(0)$; the remaining cases follow by an analogous argument applied to
the shifts of $G_\bbU$. By adjunction, choice of $G_\bbU = i^*G_\bbX$,
and substituting \eqref{eq:smashing}, we thus have to show that there
exists an integer $s(0)$ such that for all $s> s(0)$ and $\F =
\bfR{i}_*\cO_\bbU \in \bfD(\Qcoh\bbX)$:
\begin{equation}\label{eq:claim} \Hom(G_{\bbX}, \F \Lotimes X_s)
\iso \Hom(G_{\bbX}, \F \Lotimes X_{s+1}).
\end{equation} Note that the class of objects $\F \in \bfD(\Qcoh\bbX)$
for which \eqref{eq:claim} holds is closed under retracts and
arbitrary direct sums as $G_\bbX$ is compact and $\Lotimes$ commutes
with direct sums.

The object $\cO_\bbU$ is perfect on $U$, so we may invoke Neeman's result
\cite[Thm.~0.18]{Ne2017}: there exist $A_0,B_0,N_0$ with $A_0 \le B_0$
such that $\bfR{i}_*\cO_\bbU \in \sgen{G_\bbX}{N_0}{A_0}{B_0}$, i.e.,
$\bfR{i}_*\cO_\bbU$ can be built from the collection $(\Sigma^k
G_\bbX)_{A_0\le k \le B_0}$ using direct sums, retracts, and at most
$N_0$ extensions; we refer to \cite{Ne2017} for the precise
definition of $\sgen{G_\bbX}{N}{A}{B}$. For $\F = \Sigma^kG_\bbX$ and
writing $G_{\bbX}^{\vee}$ for the dual of $G_{\bbX}$, the Cauchy
property applied to $C = G_{\bbX} \Lotimes
\Sigma^{-k}G_{\bbX}^{\vee}$ provides an integer $t_k$ such that for
all $s>t_k$ there is an isomorphism
\[ \Hom(G_{\bbX} \Lotimes \Sigma^{-k}G_{\bbX}^{\vee}, X_s)
\iso \Hom(G_{\bbX} \Lotimes \Sigma^{-k}G_{\bbX}^{\vee},
X_{s+1}).
\] Set $f(1,A,B) = \max\{t_k\mid k \in [A,B]\}$, then \eqref{eq:claim}
holds  for all $s>f(1,A,B)$ and $\F \in \sgen{G_{\bbX}}{1}{A}{B}$.

We will proceed by induction on $N$, proving the following claim: for
any $N \ge 1$ and any $A\le B$ there exists an integer $f(N,A,B)$ such
that for all $s>f(N,A,B)$  and  $\F \in
\sgen{G_{\bbX}}{N}{A}{B}$ there is an isomorphism
\[ \Hom(G_{\bbX}, \F \Lotimes X_s) \iso \Hom(G_{\bbX}, \F \Lotimes
  X_{s+1}).\] This will then imply the existence of
$s(0):=f(N_0,A_0,B_0)$.

We have just checked that the claim holds for $N=1$ and arbitrary
$A\le B$. Assume the claim has been proven for $N \ge 1$ and let $\F
\in \sgen{G_{\bbX}}{N+1}{A}{B}$, i.e., $\F$ is a retract of an object
$\F'$ which fits in a triangle
\[ \E \longrightarrow \F' \longrightarrow \G \longrightarrow \Sigma\E
\] with $\E \in \sgen{G_{\bbX}}{1}{A}{B}$ and $\G \in
\sgen{G_{\bbX}}{N}{A}{B}$. We thus obtain a morphism of exact
sequences
\[ \resizebox{\textwidth}{!}{ \xymatrix{ \Hom(G_{\bbX}, \Sigma^{-1}\G
\Lotimes X_s) \ar[r] \ar[d]_{\alpha_1} & \Hom(G_{\bbX}, \E \Lotimes
X_s) \ar[r] \ar[d]_{\alpha_2} & \Hom(G_{\bbX}, \F' \Lotimes X_s)
\ar[r] \ar[d]_{\alpha_3} & \Hom(G_{\bbX}, \G \Lotimes X_s) \ar[r]
\ar[d]_{\alpha_4} & \Hom(G_{\bbX}, \Sigma\E \Lotimes X_s)
\ar[d]_{\alpha_5} \\ \Hom(G_{\bbX}, \Sigma^{-1}\G \Lotimes X_{s+1})
\ar[r] & \Hom(G_{\bbX}, \E \Lotimes X_{s+1}) \ar[r] & \Hom(G_{\bbX},
\F' \Lotimes X_{s+1}) \ar[r] & \Hom(G_{\bbX}, \G \Lotimes X_{s+1})
\ar[r] & \Hom(G_{\bbX}, \Sigma\E \Lotimes X_{s+1}).  }}
\] Note that $\Sigma^{-1}\G \in \sgen{G_{\bbX}}{N}{A-1}{B-1}$ and
$\Sigma\E \in \sgen{G_{\bbX}}{1}{A+1}{B+1}$. Therefore, if we set
\[ f(N+1,A,B) = \max\{f(1,A,B),f(1,A+1,B+1),f(N,A,B),f(N,A-1,B-1)\},
\] then the morphisms $\alpha_1, \alpha_2, \alpha_4, \alpha_5$ in the
above diagram are isomorphisms, hence so is $\alpha_3$ by the
five lemma. Consequently, $f(N+1,A,B)$ has the desired properties, and
we conclude by induction.
\end{proof}

\begin{lem}\label{lem:coherent}
  If $X \in \bfD(\Qcoh\bbX)$ is the
homotopy colimit of a Cauchy sequence $(X_s)_s$ in $\bfD^\per(\bbX)$, then $X$
has coherent cohomology sheaves.
\end{lem}
\begin{proof} Suppose $\bbU$ is an open affine subscheme of ${\bbX}$
  and denote by $i\colon\bbU \to {\bbX}$ the corresponding
  inclusion. By Lemma~\ref{lem:pullback}, $(i^*X_s)_s$ is a Cauchy
  sequence in $\bfD^\per(\bbU)$, so for any $n$ the sequence
\[ (H^n(i^*X_s))_s \cong (\Hom(\Sigma^{-n}\cO_{\bbU},X_s))_s
\] stabilises for $s \gg 0$, where the isomorphism is because $\bbU$
is affine. It follows that the homotopy colimit $i^*X$ of $(i^*X_s)_s$
has coherent cohomology. Since coherence is a local property, the
cohomology sheaves of $X$ are coherent.
\end{proof}

Next we identify the objects having bounded cohomology,
using the following criterion. We are grateful to Amnon Neeman for
suggesting a proof.

\begin{lem}\label{le:scheme-bounded} An object $X$ in
$\bfD(\Qcoh\bbX)$ belongs to $\bfD^b(\Qcoh\bbX)$ if and only if for
all compact $C$ we have $\Hom(C,\Si^nX)=0$ for $|n|\gg 0$.
\end{lem}
\begin{proof}[Proof \emph{(Neeman)}]
  Fix a compact generator $G=G_\bbX$ and an
  object $X$ in $\bfD(\Qcoh\bbX)$.

  If $X$ is in $\bfD^b(\Qcoh\bbX)$, then the fact that
$\Hom(G,\Si^nX)=0$ for $|n|\gg 0$ is standard; see
\cite[\href{http://stacks.math.columbia.edu/tag/09M2}{\S09M2}]{stacks-project}.

If $\Hom(G,\Si^nX)=0$ for $n\gg 0$, then $X$ is in
$\bfD^-(\Qcoh\bbX)$, by \cite[Thm.~4.2]{LN2007}.

Set $\T:=\bfD(\Qcoh\bbX)$. The object $G$ generates a t-structure by
setting for $n\in\bbZ$
\begin{align*} \T_G^{>n}&:=\{X\in\T\mid \Hom(G,\Si^i X)=0\textrm{ for
all } i\le n\} \intertext{and} \T_G^{\le n}&:=\{X\in\T\mid
\Hom(X,\T_G^{>n})=0\}.
\end{align*} It follows from \cite[Ex.~3.6]{Ne2018} that this
t-structure is \emph{equivalent} to the standard t-structure on $\T$
in the following sense: there exist $p\ge q$ in $\bbZ$ such that
\[\T_G^{>p}\subseteq \T^{>0}\subseteq \T_G^{>q}.\] Thus
$\Hom(G,\Si^nX)=0$ for $n\ll 0$ implies
\[X\in \bigcup_{n\in\bbZ}\T_G^{>n}=
\bigcup_{n\in\bbZ}\T^{>n}=\bfD^+(\Qcoh\bbX).\] Note that the proof in
\cite{Ne2018} requires the scheme to be quasi-compact and separated.
\end{proof}

\begin{proof}[Proof of Theorem~\ref{thm:intrinsic}] First assume that
  $X$ belongs to $\bfD^b(\coh\bbX)$.  Then the implication (1)
  $\Rightarrow$ (3) of Proposition~\ref{pr:coherent-cohom} shows that
  $X$ can be written as the homotopy colimit of a Cauchy sequence in
  $\bfD^\per(\bbX)$, keeping in mind the comments about
  pseudo-coherent objects before Proposition~\ref{pr:coherent-cohom}.
  Condition~(2) follows from Lemma~\ref{le:scheme-bounded}.

  Conversely, suppose $X$ is the homotopy colimit of a Cauchy sequence
  in $\bfD^\per(\bbX)$ satisfying Condition~(2). It follows from
  Lemma~\ref{le:scheme-bounded} that $X$ has bounded cohomology, while
  Lemma~\ref{lem:coherent} guarantees that $X$ has coherent cohomology
  sheaves, hence $X$ belongs to $\bfD^b(\coh\bbX)$.
\end{proof}

\section{Morphic enhancements and triangle completions\\[.5em]
  by Bernhard Keller}

For a category $\C$, we denote by $\M\C$ the category of morphisms
$f\colon X_1 \to X_0$ of $\C$. Let $\A$ be an abelian category and $\D\A$ its
derived category. We have the cone functor
\[
\D\M\A \to \D\A 
\]
taking a morphism of complexes $f$ to its mapping cone. It allows
us to capture the notion of standard triangle and of coherent morphism
between standard triangles. In this appendix,
we recall the axiomatization of the links between $\D\M\A$ and $\D\A$
given in Section~6 of \cite{Ke1991} and apply it to the construction of
triangulated completions of triangulated categories. Our treatment
is slightly different from that of \cite{Ke1991} because we work with
triangulated categories instead of suspended categories.
      
\subsection{Morphic enhancements} 
\label{ss:epivalence-and-recollement}
Let $\T$ and $\T_1$ be triangulated
categories and 
\[
Q_0\colon \T\to\T_1
\]
a fully faithful triangle functor admitting a
left adjoint $P_0$ and a right adjoint $P_1$. We define an additive functor
\[
M\colon \T_1 \to \M\T\; , \; X \mapsto MX=(\alpha X\colon P_1 X \to P_0 X)
\]
by requiring that $Q_0 \alpha X$ equals the composition of the adjunction morphisms
\[
Q_0 P_1 X \to X \to Q_0 P_0 X.
\]
For example, with the above notations, we can take $\T=\D\A$ and $\T_1=\D\M\A$.
We identify the objects of $\D\M\A$ with morphisms of complexes. With this
convention, the functor $Q_0$ takes a complex $X$ to its identity morphism and
the functors $P_i$ take a morphism of complexes $f\colon X_1 \to X_0$ to 
$X_i$, $i=0,1$. Then the functor $M\colon \D\M\A \to \M\D\A$ takes 
a morphism of complexes $f\colon X_1 \to X_0$ to its image in the category
of morphisms of the derived category $\D\A$ (each morphism of complexes
yields a morphism in the derived category).

Recall that the triple $(P_0, Q_0, P_1)$ yields a canonical recollement
\cite{BBD1982}
\begin{equation} \label{recollement1}
\begin{tikzcd}
  \T\arrow{rr}[description]{Q_0} &&\T_1
  \arrow[yshift=-1.5ex]{ll}{P_1}
  \arrow[yshift=1.5ex]{ll}[swap]{P_0}
  \arrow{rr}[description]{P_2'} &&\T_1/\T\,.
  \arrow[yshift=-1.5ex]{ll}{Q_2'}
  \arrow[yshift=1.5ex]{ll}[swap]{Q_1'}
\end{tikzcd}
\end{equation}
Recall that a functor is an \emph{epivalence} if it is conservative
(i.e.~it detects isomorphisms), full and essentially surjective.

\begin{thm} \label{thm:epi-rec} The following are equivalent:
\begin{itemize}
\item[(i)] The functor $M\colon \T_1 \to \M\T$ is an epivalence.
\item[(ii)] In the above recollement, we have $\Ker P_0\subseteq (\Ker P_1)^\perp$
and the composition $P_1 Q_1'$ is an equivalence.
\end{itemize}
\end{thm}

\begin{defn} The triangle
functor $Q_0\colon \T \to \T_1$ is a \emph{morphic enhancement
of $\T$} if the equivalent conditions of the theorem hold.
\end{defn}

For example, it is easy to check that the functor
$M\colon \D\M\A\to\M\D\A$ is an epivalence if $\A$ is an abelian
category. For more examples, we refer to Section~\ref{ss:examples}
below.

\begin{proof}[Proof of Theorem~\ref{thm:epi-rec}]
We prove the implication from (i) to (ii). We start by constructing
a right adjoint $P_{1\rho}$ of $P_1$. Let $Y$ be an object of $\T$. Since $M$
is an epivalence, we can find an object $P_{1\rho}Y$ and an isomorphism
in $\M\T$
\[
\xymatrix{
P_1 P_{1\rho} Y \ar[r] \ar[d]_{\phi Y} & P_0 P_{1\rho} Y \ar[d]\\
Y \ar[r] & \;0.}
\]
For $X$ in $\T_1$, consider the map
\[
\T_1(X,P_{1\rho} Y) \to \T(P_1X,Y)
\]
taking $f$ to $(\phi Y)(P_1 f)$. Since $M$ is full, it is surjective. Suppose that 
$f\colon X \to P_{1\rho} Y$ is in the kernel. Form a triangle
\[
\xymatrix{
X \ar[r]^-f & P_{1\rho} Y \ar[r]^-g & Z \ar[r] & \Sigma X.}
\]
Its image under $M$ is the morphism of triangles
\[
\xymatrix{
P_1 X \ar[r]^-{0} \ar[d] & P_1 P_{1\rho} Y \ar[r]^-{P_1 g} \ar[d] 
& P_1 Z \ar[r] \ar[d] & \Sigma P_1 X \ar[d] \\
P_0 X \ar[r] & 0 \ar[r] & P_0 Z \ar[r] & \Sigma P_0 X.}
\]
This shows that $Mg$ admits a retraction in $\M\T$. Since $M$ is full and
conservative, this implies that $g$ admits a retraction and $f$ vanishes.
Thus we have the right adjoint $P_{1\rho}$ of $P_1$. Since $P_1$ is a
localization functor, $P_{1\rho}$ is fully faithful. By construction, its image is
$\Ker P_0$. Therefore, the functors $P_1$ and $P_{1\rho}$ induce
quasi-inverse equivalences between $\T$ and $\Ker P_0$. We also
know that $Q_1'$ and $P_2'$ induce quasi-inverse equivalences between
$\T_1/\T$ and $\Ker P_0$. Thus, the functor $P_1 Q_1'$ is an equivalence.
Now let $X$ be in $\Ker P_1$, $Y$ in $\Ker P_0$ and let $f\colon X\to Y$ be
a morphism. Form a triangle
\[
\xymatrix{X \ar[r]^f & Y \ar[r]^g & Z \ar[r] & \Sigma X.}
\]
Its image under $M$ is the morphism of triangles
\[
\xymatrix{0 \ar[r] \ar[d] & P_1 Y \ar[d] \ar[r] & P_1 Z \ar[d] \ar[r] & \Sigma 0 \ar[d] \\
P_0 X \ar[r] & 0 \ar[r] & P_0 Z \ar[r] & \Sigma P_0 X.}
\]
It follows that $Mg$ admits a retraction. Since $M$ is full and conservative,
$g$ admits a retraction and $f$ vanishes.

We prove the implication from (ii) to (i). Let $X$ be in $\T_1$ and $Y$ in $\T$.
We have
\[
\T(P_1 X,Y) \iso \T(P_1 X, P_1 Q_1' (P_1 Q_1')^{-1} Y) \iso 
\T(Q_0 P_1 X, Q_1' (P_1Q'_1)^{-1}Y).
\]
We have the triangle
\[
\xymatrix{
Q_0 P_1 X \ar[r] & X \ar[r] & Q_2' P_2' X \ar[r] & \Sigma Q_0 P_1 X.}
\]
Since there are no nonzero morphisms from $\Ker P_1 = \Im Q_2'$ to
$\Ker P_0=\Im Q_1'$, it follows that we have a bijection
\[
\T(P_1 X,Y) \iso \T(X, Q_1' (P_1 Q_1')^{-1} Y).
\]
Thus, the functor $P_1$ admits the right adjoint $Q_1=Q_1' (P_1 Q_1')^{-1}$.
Since $P_1$ is a localization functor, $Q_1$ is fully faithful. We have
$\Im Q_1 =\Im Q_1' =\Ker P_0$. We prepare for the proof of the fullness of $M$.
Let $X$ be in $\T_1$. We form the triangle
\begin{equation} \label{eq:inj-res}
\xymatrix{
X \ar[r] & Q_0 P_0 X \oplus Q_1 P_1 X \ar[r] & Y \ar[r] & \Sigma X
}
\end{equation}
over the morphism whose components are the adjunction morphisms.
The adjunctions yield a canonical isomorphism
\[
P_1 Q_0 P_0 X \oplus P_1 Q_1 P_1 X \iso  P_0 X \oplus P_1 X.
\]
Using this we see that the image of the triangle (\ref{eq:inj-res}) under $M$ is
the morphism of triangles
\[
\xymatrix{
P_1 X \ar[r]^-{\smatrix{\alpha X\\ \id}} \ar[d]_{\alpha X} &
P_0 X \oplus P_1 X \ar[d]^{\smatrix{\id&0}} \ar[r]^-{\smatrix{\beta&\gamma}} &
P_1 Y \ar[d] \ar[r] &
\Sigma P_1 X \ar[d]\\
P_0 X \ar[r]^\id & P_0 X \ar[r] & 0 \ar[r] & \Sigma P_0 X.
}
\]
In particular, $Y$ belongs to $\Ker P_0$ and
$\beta\colon P_0 X \iso P_1 Y$ yields a canonical isomorphism
$Q_1 P_0 X \iso Y$. Thus the triangle (\ref{eq:inj-res}) is isomorphic
to a functorial triangle
\begin{equation} \label{eq:functorial-triangle}
\xymatrix{
X \ar[r] & Q_0 P_0 X \oplus Q_1 P_1 X \ar[r] & Q_1 P_0 X \ar[r] & \Sigma X.
}
\end{equation}
Let $Z$ be in $\T_1$. If we apply $\T_1(Z,?)$ to this triangle and use the
adjunctions, we obtain a bifunctorial exact sequence
\[
\xymatrix{
\T(\Sigma P_1 Z, P_0 X) \ar[r] & \T_1(Z,X) \ar[r] & \M\T(MZ,MX) \ar[r] & 0.
}
\]
This shows in particular that $M\colon \T_1 \to \M\T_0$ is full. We claim that its
kernel is an ideal of square $0$. Indeed, suppose that $g\colon X \to Y$ and 
$f\colon Y\to Z$
belong to the kernel. Consider the morphisms of triangles
\[
\xymatrix{
\Sigma^{-1} Q_1 P_0 X \ar[r] \ar[d]^0 & X \ar[d]^g \ar[r] & 
                           Q_0 P_0 X \oplus Q_1 P_1 X \ar[d]^0 \ar[r] & Q_1 P_0 X \ar[d] \\
\Sigma^{-1} Q_1 P_0 Y \ar[r] \ar[d]^0 & Y \ar[d]^f \ar[r] &
	 		Q_0 P_0 Y \oplus Q_1 P_1 Y \ar[d]^0 \ar[r] & Q_1 P_0 Y \ar[d]\\
\Sigma^{-1} Q_1 P_0 Z \ar[r] & Z \ar[r] &
			Q_0 P_0 Z \oplus Q_1 P_1 Z \ar[r] & Q_1 P_0 Z.
}
\]
Clearly $g$ factors through $\Sigma^{-1}Q_1 P_0 Y \to Y$ and $f$
factors through 
\[
Y \to Q_0 P_0 Y\oplus Q_1 P_1 Y.
\]
It follows that $fg$ vanishes. Since $M$ is full, it follows that it is conservative.
Let us show that it is essentially surjective. Let $f\colon X_1 \to X_0$ be an
object of $\M\T$ and consider the componentwise split short exact
sequence of $\M\T$
\[
\xymatrix{
0 \ar[r] & X_1 \ar[d]^f \ar[r]^-{\smatrix{\id\\ f}} & X_1 \oplus X_0
\ar[d]^{\smatrix{0&\id}} \ar[r]^-{\smatrix{-f&\id}} &
X_0 \ar[r] \ar[d] & 0 \\
0 \ar[r] & X_0 \ar[r]^\id & X_0 \ar[r] & 0 \ar[r] & 0.
}
\]
The middle and the right hand term lift respectively to $Q_1 X_1 \oplus Q_0 X_0$
and $Q_1 X_0$. Since $M$ is full, the morphism between the middle and
the right hand term lifts to a morphism $g\colon Q_1 X_1 \oplus Q_0 X_0 \to Q_1 X_0$.
It is easy to see that $f\colon X_1 \to X_0$ is isomorphic to $MY$, where $Y$ is
defined by the triangle
\[
\xymatrix{
\Sigma^{-1} Q_1 X_0 \ar[r] & Y \ar[r] & Q_1 X_1\oplus Q_0 X_0 \ar[r]^-g & Q_1 X_0.}\qedhere
\]
\end{proof}

\subsection{Examples} \label{ss:examples} As we have already seen,
if $\A$ is abelian, then the derived category $\D\M\A$ yields a morphic
enhancement of $\D\A$.

More generally, if $\T$ is an algebraic triangulated
category, i.e.~triangle equivalent to the stable category $\underline{\E}$
of a Frobenius category $\E$, then it admits a morphic enhancement given by the functor
\[
X \mapsto (\id\colon X \to X)
\]
from $\underline{\E}$ to the stable category $\underline{\I\E}$ of the category
$\I\E$ of inflations $X_1 \to X_0$ of $\E$ endowed with the class of short
exact sequences inducing conflations in the two components and in the 
cokernel, cf.\ Example b) of Section~6.1 in \cite{Ke1991}.

More generally,
if $\T$ is the base category of a (strong) stable derivator $\mathbb{D}$
in the sense of \cite{Groth2013},
then the value of $\mathbb{D}$ on the index category $\{0<1\}$ yields
a morphic enhancement of $\T$. Since the homotopy category of each
combinatorial stable model category is the base of a stable derivator
(cf.~\cite{Cisinski2003} \cite{Groth2013}), the triangulated categories
arising commonly in algebra and topology admit morphic enhancements.

More generally, if $\T$ is the base of an epivalent tower of triangulated
categories in the sense of \cite{Ke1991}, then $\T$ has a morphic
enhancement given by the first floor of the tower. 

\subsection{Properties} \label{ss:properties} Let $Q_0\colon \T \to \T_1$ be
a morphic enhancement in the sense of Section~\ref{ss:epivalence-and-recollement}.
Part b) of the following proposition shows that
the underlying additive categories of $\T$ and $\T_1$
and the additive functor $Q_0\colon \T \to \T_1$ determine the triangulated
structure of $\T$.

\begin{prop} \label{prop:properties}
\begin{itemize}
\item[a)] There is an infinite sequence of adjoint functors
\[
\ldots \dashv P_n \dashv Q_n \dashv P_{n+1} \dashv Q_{n+1} \dashv \ldots\ , n\in\mathbb{Z}.
\]
We have natural isomorphisms $\Sigma Q_n \iso Q_{n+3}$ and 
$\Sigma P_n \iso P_{n-3}$ for all integers $n$. Each $Q_n\colon \T \to \T_1$
is a morphic enhancement with associated recollement
\[\begin{tikzcd}
  \T\arrow{rr}[description]{Q_n} &&\T_1
  \arrow[yshift=-1.5ex]{ll}{P_{n+1}}
  \arrow[yshift=1.5ex]{ll}[swap]{P_n}
  \arrow{rr}[description]{P_{n+2}} &&\T\,.
  \arrow[yshift=-1.5ex]{ll}{Q_{n+2}}
  \arrow[yshift=1.5ex]{ll}[swap]{Q_{n+1}}
\end{tikzcd}
\]
\item[b)] We have a canonical isomorphism $\Sigma \iso P_{-1} Q_1$.
For each $X\in \T_1$, there is a functorial triangle
\[
\xymatrix{ 
P_1 X \ar[r]^{\alpha X} & P_0 X \ar[r] & P_{-1} X \ar[r] & \Sigma P_1 X,
}
\]
where the first two morphisms are given by the adjunctions and the third one
is the composition
\[
P_{-1} X \to P_{-2} X \iso \Sigma P_1 X.
\]
Each triangle of $\T$ is isomorphic to a triangle of this form.
\item[c)] The kernel of $M\colon \T_1 \to \M\T$ is an ideal of square zero.
For $X,Y\in\T_1$, there is a bifunctorial exact sequence
\[
\T(\Sigma P_1 X, P_0 Y) \to \T_1(X,Y) \to \M\T(MX,MY) \to 0.
\]
\end{itemize}
\end{prop}

\begin{proof} a) In the proof of Theorem~\ref{thm:epi-rec}, we have constructed
the adjoint $Q_1$ as $Q_1' (P_1 Q_1')^{-1}$. Since $P_1 Q_1'$ is an equivalence
and $Q_1'$ admits the right adjoint $P_2'$, the functor $Q_1$ admits a right
adjoint $P_2$. We have $\Im Q_1 = \Ker P_0$. Whence a short exact sequence
\[
\xymatrix{
0\ar[r] & \T \ar[r]^{Q_1} & \T_1 \ar[r]^{P_0} & \T \ar[r] & 0
}
\]
and a recollement
\[\begin{tikzcd}
  \T\arrow{rr}[description]{Q_1} &&\T_1
  \arrow[yshift=-1.5ex]{ll}{P_2}
  \arrow[yshift=1.5ex]{ll}[swap]{P_1}
  \arrow{rr}[description]{P_0} &&\T\,.
  \arrow[yshift=-1.5ex]{ll}{Q_0}
  \arrow[yshift=1.5ex]{ll}[swap]{Q_{-1}}
\end{tikzcd}
\]
We need to show that $Q_1\colon \T \to \T_1$ is a morphic enhancement. We have
$\Ker P_2=\Im Q_0$ and for $X\in \T$ and $Y\in \Ker P_1$, we have
\[
\T_1(Q_0 X,Y) \iso \T(X,P_1 Y)=0.
\]
Thus, we have $\Ker P_1 \subseteq (\Ker P_2)^\perp$. We need to check
that $P_2 Q_{-1}$ is an equivalence. For $X\in \T$, we have the triangle
\[
\xymatrix{Q_1 P_2 Q_{-1} X \ar[r] & Q_{-1}X \ar[r] & Q_0 P_0 Q_{-1} X \ar[r] &
\Sigma Q_1 P_2 Q_{-1} X.}
\]
By applying $P_1$ to this triangle we find the triangle
\[
\xymatrix{P_2 Q_{-1} X \ar[r] & 0 \ar[r] & P_0 Q_{-1} X \ar[r] & \Sigma P_2 Q_{-1} X}
\]
and therefore we have $X \iso \Sigma P_2 Q_{-1} X$. Therefore, the functor
$P_2 Q_{-1}\colon \T \to \T$ is an equivalence isomorphic to $\Sigma^{-1}$.
Thus, the functor $Q_1\colon \T \to \T_1$ is a morphic enhancement. By induction,
we get a sequence of adjoints
\[
P_0 \dashv Q_0 \dashv P_1 \dashv Q_1 \dashv \ldots \dashv P_n \dashv Q_n \dashv \ldots\ .
\]
Since our assumption is self-dual, we also get a sequence of adjoints
\[
\ldots \dashv P_{-n} \dashv Q_{-n} \dashv \ldots \dashv P_0 \dashv Q_0.
\]
We have already constructed an isomorphism $P_2 Q_{-1} \iso \Sigma^{-1}$
and we have $\Sigma^{-1} \iso P_2 \Sigma^{-1} Q_2$. Whence an
isomorphism $P_2 \Sigma Q_{-1} \iso P_2  Q_2$. We have
$\Im Q_{-1} = \Ker P_1 \subseteq (\Ker P_2)^\perp$. By the recollement
\[\begin{tikzcd}
  \T\arrow{rr}[description]{Q_0} &&\T_1
  \arrow[yshift=-1.5ex]{ll}{P_{1}}
  \arrow[yshift=1.5ex]{ll}[swap]{P_0}
  \arrow{rr}[description]{P_{2}} &&\T
  \arrow[yshift=-1.5ex]{ll}{Q_{2}}
  \arrow[yshift=1.5ex]{ll}[swap]{Q_{1}}
\end{tikzcd}
\]
we also have $\Im Q_2 =\Ker P_1 \subseteq (\Ker P_2)^\perp$. Since the
restriction of $P_2$ to $(\Ker P_2)^\perp$ is an equivalence, we get an
isomorphism $\Sigma Q_{-1} \iso Q_2$. By induction, we get
$\Sigma Q_n \iso Q_{n+3}$ and by adjunction $\Sigma^{-1} P_n \iso P_{n+3}$
for all integers $n$. 

b) We have seen in the proof of a) that there is a canonical isomorphism
$P_2 Q_{-1} \iso \Sigma^{-1}$. By passing to the left adjoints we get
an isomorphism $Q_1 P_{-1} \iso \Sigma$. As it follows from the proof
of a), we have a recollement
\[\begin{tikzcd}
  \T\arrow{rr}[description]{Q_{-1}} &&\T_1
  \arrow[yshift=-1.5ex]{ll}{P_{0}}
  \arrow[yshift=1.5ex]{ll}[swap]{P_{-1}}
  \arrow{rr}[description]{P_{1}} &&\T\,.
  \arrow[yshift=-1.5ex]{ll}{Q_{1}}
  \arrow[yshift=1.5ex]{ll}[swap]{Q_{0}}
\end{tikzcd}
\]
Thus, for $X\in \T_1$, we have the functorial triangle
\[
\xymatrix{
Q_0 P_1 X \ar[r] & X \ar[r] & Q_{-1} P_{-1} X \ar[r] & \Sigma Q_0 P_1 X.}
\]
By applying $P_0$ to this triangle we get the functorial triangle
\[
\xymatrix{
P_1 X \ar[r] & P_0 X \ar[r] & P_{-1} X \ar[r] & \Sigma P_1 X.}
\]
Since $M$ is essentially surjective, each triangle of $\T$ is isomorphic
to a triangle of this form.

c) This was already shown in the proof of the implication from (ii) to (i) in
Theorem~\ref{thm:epi-rec}.
\end{proof}

We keep the assumptions on $Q_0 \colon\T \to \T_1$.
\begin{defn} \label{def:coherent}
A \emph{standard triangle} is a triangle of $\T$
\[
\xymatrix{P_1 X \ar[r] & P_0 X \ar[r] & P_{-1} X \ar[r] & \Sigma P_0 X}
\]
associated with an object $X$ of $\T_1$. A \emph{coherent morphism
between standard triangles} is a morphism
\[
\xymatrix{
P_1 X \ar[d] \ar[r] & P_0 X \ar[d] \ar[r] & P_{-1} X \ar[d] \ar[r] & \Sigma P_0 X \ar[d]\\
P_1 Y \ar[r] & P_0 Y \ar[r] & P_{-1} Y \ar[r] & \Sigma P_1 Y}
\]
induced by a morphism $X \to Y$ of $\T_1$.
\end{defn}

 The following lemma will
be crucial in checking the axioms of a triangulated category for a completion.

\begin{lem} \label{lemma:octahedron}
Let $f\colon X \to Y$ be a morphism of $\T_1$. Let 
\[
X'=Q_{-1}P_{-1}X \oplus Q_0 P_0 X\oplus Q_1 P_1 X
\]
and let $f'\colon X \to Y\oplus X'$ be the morphism whose components are $f$
and the adjunction morphisms. Let
\[
\xymatrix{
X \ar[r]^-{f'} & Y\oplus X' \ar[r] & Z \ar[r] & \Sigma X}
\]
be a triangle. Then the standard triangle
\[
\xymatrix{
P_1 Z \ar[r] & P_0 Z \ar[r] & P_{-1} Z \ar[r] & \Sigma P_1 Z}
\]
is isomorphic to the mapping cone \cite[Section~1.3]{Ne2001}
over the morphism of triangles
\[
\xymatrix{
P_1 X \ar[d]^{P_1 f} \ar[r] & P_0 X \ar[d]^{P_0 f} \ar[r] & P_{-1} X \ar[d]^{P_{-1}f} \ar[r] &
\Sigma P_1 X \ar[d]^{\Sigma P_1 f} \\
P_1 Y \ar[r] & P_0 Y \ar[r] & P_{-1} Y \ar[r] & \Sigma P_1 Y.}
\]
\end{lem}

\begin{proof} For an object $X$ of $\T_1$, let
\[
\xymatrix{
P_1 X \ar[r]^{\alpha X} & P_0 X \ar[r]^{\beta X} & P_{-1} X \ar[r]^{\gamma X} &
\Sigma P_0 X}
\]
be the standard triangle and let $TX$ be the complex obtained by glueing the
following triangles along their boundaries
\[
\xymatrix{
\Sigma^p P_1 X \ar[r]^{\alpha_p X} &
\Sigma^p P_0 X \ar[r]^{\beta_p X} &
\Sigma^p P_{-1} X \ar[r]^{\gamma_p X} &
\Sigma^{p+1} P_1 X}\;, \;\; p\in\mathbb{Z}\;,
\]
where
$\alpha_p X=(-1)^p \Sigma^p \alpha X$, $\beta_p X=(-1)^p \Sigma^p \beta X$
and $\gamma_p X = (-1)^p \Sigma^p \gamma X$. We have to show that
$TZ$ is isomorphic to the mapping cone over the morphism $Tf\colon TX \to TY$.
For a complex $C=(C^p,d^p)$, let $IC$ be the complex with components
$C^p\oplus C^{p+1}$ and the differential
\[
\smatrix{0 & \id \\ 0 & 0} \colon
C^p \oplus C^{p+1} \to C^{p+1}\oplus C^{p+2}
\]
and let $i_C\colon C \to IC$ be the morphism of complexes with the components
\[
\smatrix{\id \\ d^p} \colon C^p \to C^p\oplus C^{p+1}.
\]
One checks easily that $TX'$ is canonically isomorphic to $ITX$ and that
the morphism $X \to X'$ whose components are the adjunction morphisms
induces the morphism $i_{TX}\colon TX \to ITX$. Thus, the morphism 
$f\colon X \to Y\oplus X'$
induces the morphism
\[
\smatrix{Tf \\ i_{TX}}\colon TX \to TY\oplus ITX.
\]
Notice that this is a componentwise split monomorphism whose cokernel
is canonically isomorphic to the mapping cone over $Tf\colon TX \to TY$.
Now the triangle
\[
\xymatrix{
X \ar[r]^-{f'} & Y\oplus X' \ar[r] & Z \ar[r] & \Sigma X}
\]
yields a componentwise split exact sequence
\[
\xymatrix{
0 \ar[r] & TX \ar[r] & T(Y\oplus X') \ar[r] & TZ \ar[r] & 0.}
\]
It follows that $TZ$ is canonically isomorphic to the cone over $Tf$.
\end{proof}

\subsection{Morphic functors, compact objects} 
Let $Q_0\colon \calS \to \calS_1$ and $Q_0\colon \T\to \T_1$
be morphic enhancements. A \emph{morphic functor} $\calS \to \T$ is given by
a square of triangle functors
\[
\xymatrix{\calS \ar[d]_F \ar[r]^{Q_0} & \calS_1 \ar[d]^{F_1} \\
\T \ar[r]^{Q_0} & \T_1}
\]
commutative up to given isomorphism such that the canonical
morphisms
\[
P_0 F_1 \to F P_0 \quad\text{and}\quad FP_1 \to P_1 F_1
\]
are invertible. Examples of morphic functors are provided by
morphisms of epivalent towers of triangulated categories
\cite{Ke1991} and by morphisms of stable derivators \cite{Groth2013}.

\begin{lem} \label{lemma:morphic}
If $F\colon \calS\to\T$ is a morphic functor, we have canonical
isomorphisms
\[
P_n F_1 \iso F P_n \quad\text{and}\quad Q_n F \iso F_1 Q_n
\]
for all integers $n$.
\end{lem}

\begin{proof} Since $Q_{-1}\colon \calS \to \calS_1$ and $Q_{-1}\colon\T\to \T_1$ 
are again morphic enhancements, it is enough to show the claim for
$Q_{-1}$ and $P_{-1}$. Indeed, by induction it will then follow for
$Q_n$ and $P_n$ for all $n<0$ and by duality for all $n\geq 0$. The image
of the canonical morphism $Q_{-1} F \to F_1 Q_{-1}$ under $P_0$ fits
into the commutative square
\[
\xymatrix{
F \ar[d]_\sim \ar[r]^-{\sim} & F P_0 Q_{-1} \ar[d]^\sim \\
P_0 Q_{-1} F \ar[r] & P_0 F_1 Q_{-1}.}
\]
Thus, it is invertible. The image of $Q_{-1} F \to F_1 Q_{-1}$ under $P_1$ is
the identity of the zero object. Since $M$ is conservative, it follows that
the morphism $Q_{-1} F \to F_1 Q_{-1}$ is invertible. Now consider the
canonical morphism
\[
P_{-1} F_1 \to F P_{-1}.
\]
For each $X$ of $\calS_1$, it fits into a morphism of triangles
\[
\xymatrix{
P_1 F_1 X \ar[r] \ar[d]^\sim & P_0 F_1 X \ar[r]\ar[d]^\sim &
P_{-1} F_1 X \ar[d] \ar[r] & \Sigma P_1 F_1 X \ar[d]^\sim\\
FP_1 X \ar[r] & FP_0 X \ar[r] & FP_{-1} X \ar[r] & \Sigma FP_1 X.}
\]
Thus, it is invertible. 
\end{proof}

The inclusion of the subcategory of compact objects in a compactly
generated triangulated category with a morphic enhancement is an
example of a morphic functor, as shown by the next lemma.

\begin{lem} \label{lemma:compact}
Let $Q_0\colon\T \to \T_1$ be a morphic enhancement. Then
$\T$ is compactly generated if and only if $\T_1$ is compactly generated.
In this case, an object $X$ of $\T_1$ is compact if and only if
$P_0 X$ and $P_1 X$ are compact and
the functor $Q_0$ induces a morphic enhancement 
$\T^c \to \T_1^c$, where $\T^c$ is the subcategory of compact objects.
Moreover, the inclusion $\T^c \to \T$ is a morphic functor.
\end{lem}

\begin{proof} The localization functor $P_1\colon \T_1 \to \T$ admits two successive
right adjoints. Thus it commutes with arbitrary coproducts and preserves
compactness. Therefore, if $\T_1$ is compactly generated, then so is $\T$.
For the converse implication, note that all the functors $Q_n$
commute with arbitrary coproducts and preserve compactness (for the
same reason). So if $\T$ is compactly generated, so are the
subcategories $\Im Q_0$ and $\Im Q_2$ of $\T_1$ and their inclusions
commute with arbitrary coproducts. Since each object of $\T_1$ is
an extension of an object of $\Im Q_2$ by an object of $\Im Q_0$, it
follows that $\T_1$ has arbitrary coproducts and is generated by
the compact objects in $\Im Q_2$ and $\Im Q_0$. Suppose that
$\T$ and $\T_1$ are compactly generated. The functors $P_1$ and
$P_0$ preserve compactness and so do $Q_{-1}$ and $Q_1$.
If an object $X$ of $\T_1$ has compact $P_0 X$ and $P_1 X$, then it
is itself compact because of the triangle
\[
\xymatrix{
Q_{-1} P_0 X \ar[r] & X \ar[r] & Q_1 P_1 X \ar[r] & \Sigma Q_{-1} P_0 X.}
\] 
Now suppose that $\T$ and $\T_1$ are compactly generated. 
As we have seen, the functor $Q_0\colon \T \to \T_1$ induces a functor 
$Q_0\colon \T^c \to \T_1^c$ and $P_0$ and $P_1$ induce left and right adjoints.
It is now clear that the epivalence $M\colon \T_1 \to \M\T$ induces an
epivalence $M\colon \T_1^c \to \M\T^c$ and that the inclusion 
$\T^c \to \T_1^c$ is a morphic functor.
\end{proof}

\subsection{Completion} \label{ss:completion}
Let $\T$ be a triangulated category
with a morphic enhancement $Q_0 \colon\T \to \T_1$. Let $\X$ be a class
of sequences 
\[
X_0 \to X_1 \to \ldots \to X_p \to \ldots
\]
of $\T$. We assume that the following hold for $\X$:
\begin{itemize}
\item[a)] $\X$ is stable under passage to cofinal sequences i.e. if $(X_p)$
belongs to $\X$ and $(i_p)$ is a strictly increasing sequence of positive
integers, then $(X_{i_p})$ belongs to $\X$.
\item[b)] $\X$ is stable under $\Sigma$ and $\Sigma^{-1}$ i.e. if $(X_p)$ belongs
to $\X$ so do $(\Sigma X_p)$ and $(\Sigma^{-1} X_p)$.
\item[c)] $\X$ is stable under cones i.e. if $f\colon X \to Y$ is a morphism of
sequences of $\X$ and for each $p\geq 0$, the object $Z_p$ of $\T_1$
lifts the object $f_p \colon X_p \to Y_p$ of $\M\T$ and the morphism $Z_p \to Z_{p+1}$
lifts the morphism
\[
\xymatrix{ X_p \ar[r] \ar[d] & Y_p \ar[d] \\
X_{p+1} \ar[r] & Y_{p+1}}
\]
of $\M\T$, then the sequence $(P_{-1} Z_p)$ belongs to $\X$.
\item[d)] $\X$ is \emph{phantomless} i.e. for $X$ and $Y$ in $\X$, we have
\[
{\lim_p}^1 \colim_q \T(X_p, Y_q) =0.
\]
\end{itemize}
Let $\Mod\T$ be the category of additive functors $\T^{\op}\to \Ab$.
For $X\in \T$, let $X^\wedge$ be the functor $\T(?,X)$ represented by $X$.
Define the completion $\widehat{\T_\X}$ to be the full subcategory of
$\Mod\T$ whose objects are the colimits
\[
LX = \colim_p X_p^\wedge
\]
of sequences $X$ of $\X$. Note that by Proposition~\ref{pr:completion}, 
this agrees with the definition in the main text.
Define $\Sigma\colon \widehat{\T_\X} \to \widehat{\T_\X}$
to be the functor induced by $X \mapsto \Sigma X$. For a sequence
$(Z_p)$ of $\T_1$ such that $(P_1 Z_p)$ and $(P_0 Z_p)$ belong to $\X$,
define the \emph{standard triangle} associated with $Z$ to be the 
$\Sigma$--sequence
\[
\xymatrix{
L(P_1 Z_p) \ar[r] & L(P_0 Z_p) \ar[r] & L(P_{-1} Z_p) \ar[r] & \Sigma L(P_1 Z_p).
}
\]
Thus, the standard triangles of $\widehat{\T_\X}$ are exactly the
colimits 
of sequences of coherent morphisms between standard triangles of $\T$,
cf.~Definition~\ref{def:coherent}.  Define a \emph{triangle} of
$\widehat{\T_\X}$ to be a $\Sigma$--sequence isomorphic to a standard
triangle.

\begin{thm}\label{th:morphic-completion}
  Endowed with the suspension functor $\Sigma$ and the above triangles
  the completion $\widehat{\T_\X}$ is a triangulated category.  If
  $\X$ contains the constant sequences consisting of identities only,
  we have a canonical triangle embedding $\T \to \widehat{\T_\X}$.
\end{thm}

\begin{proof} Let $X$ be a sequence in $\X$. We need to show that
\[
\xymatrix{
LX \ar[r]^{\id} & LX \ar[r] & 0 \ar[r] & \Sigma LX
}
\]
is a triangle (TR0). In fact, it is the standard triangle associated with
the sequence $(Q_0 X_p)$. Let $X$ and $Y$ be sequences in $\X$ and
let $f\colon LX \to LY$ be a morphism. After passing to cofinal sequences we
may assume that $f$ is in fact a morphism of sequences. We lift each
object $f_p\colon X_p \to Y_p$ of $\M\T$ to an object $Z_p$ of $\T_1$ and
each morphism
\[
\xymatrix{
X_p \ar[d] \ar[r]^{f_p} & Y_p \ar[d]\\
X_{p+1} \ar[r]^{f_{p+1}} & Y_{p+1}}
\]
to a morphism $Z_p\to Z_{p+1}$ of $\T_1$. Then the standard triangle
associated with $(Z_p)$ yields a triangle whose first morphism identifies
with $f\colon LX \to LY$ and we have proved TR1. Let
\[
\xymatrix{
LX \ar[d] \ar[r] & LY \ar[r] \ar[d] & LZ \ar[r] & \Sigma LX \ar[d]\\
LX' \ar[r] & LY' \ar[r] & LZ' \ar[r] & \Sigma LX'}
\]
be a commutative diagram of $\widehat{\T_\X}$ whose rows are triangles.
We will show that there is a morphism $LZ \to LZ'$ completing the diagram
to a morphism of triangles whose mapping cone is still a triangle.
This implies the rotation axiom TR2 (take $LX'=LY'=LZ'=0$), the
axiom about the missing morphism TR3 and axiom TR4' of Section~1.4
of \cite{Ne2001}. By Proposition~1.4.6 of \cite{Ne2001}, the octahedral
axiom TR4 follows. We may assume that the given triangles are the standard
triangles associated with sequences $U$ and $V$ of $\T_1$. The given
commutative diagram
\[
\xymatrix{
LP_1 U \ar[r] \ar[d] & LP_0 U \ar[d] \\
LP_1 V \ar[r] & LP_0 V}
\]
is a morphism in $\M\widehat{\T_\X}$. As in 
Example~\ref{example:morphism-category}, we see that the canonical functor
$\widehat{(\M\T)_{\M\X}} \to \M\widehat{\T_\X}$ is an equivalence, where
$\M\X$ is the class of sequences of morphisms $(X_{1p}\to X_{0p})$
with $(X_{1p})$ and $(X_{0p})$ in $\X$. 
So the given commutative diagram may be viewed as a morphism
\[
LMU \to LMV
\]
in $\widehat{(\M T)_{\M\X}}$. We claim that it suffices to lift it
to a morphism of sequences $U \to V$ of $\T_1$. Indeed, once we have
such a lift, we can form triangles
\[
\xymatrix{
U_p \ar[r] & V_p \oplus U_p' \ar[r] & W_p \ar[r] & \Sigma U_p}
\]
as in Lemma~\ref{lemma:octahedron} and morphisms of triangles
\[
\xymatrix{
U_p \ar[d]\ar[r] & V_p \oplus U_p' \ar[d]\ar[r] & W_p \ar[d]\ar[r] & \Sigma U_p\ar[d] \\
U_{p+1} \ar[r] & V_{p+1} \oplus U'_{p+1} \ar[r] & W_{p+1} \ar[r] & \Sigma U_{p+1}.}
\]
By Lemma~\ref{lemma:octahedron}, the standard triangle associated with $W$
will be the mapping cone over the morphism of standard triangles associated
with $U \to V$. The given morphism $LMU\to LMV$ identifies with an element of 
\[
\lim_p \colim_q \M\T(MU_p, MV_q).
\]
It suffices to show that the natural map
\[
\lim_p \colim_q \T_1(U_p, V_q) \to \lim_p \colim_q \M\T(MU_p,MV_q)
\]
is surjective. By part c) of Proposition~\ref{prop:properties}, for each $p\geq 0$,
we have an exact sequence
\[
\colim_q \T(\Sigma P_1 U_p, P_0 V_q) \to \colim_q \T_1(U_p, V_q) \to
\colim_q \M\T(MU_p, MV_q) \to 0.
\]
Let us abbreviate it to
\[
A_p \to B_p \to C_p \to 0.
\]
Let $A'_p$ be the image of $A_p$ in $B_p$. We have the exact sequence
\[
0 \to \lim A'_p \to \lim B_p \to \lim C_p \to {\lim}^1 A'_p.
\]
Since $\lim^1$ is right exact, we have a surjection $\lim^1 A_p\to \lim^1 A'_p$.
Since $\X$ is phantomless, the group $\lim^1 A_p$ vanishes. So
\[
\lim B_p \to \lim C_p
\]
is surjective as required.
\end{proof}

\subsection{Functoriality} The construction of the completion is functorial
with respect to morphic triangle functors. Let us spell this out: Let $\T$ and $\T'$
be skeletally small triangulated categories and $Q_0 \colon \T \to \T_1$ and
$Q_0\colon \T' \to \T'_1$ be morphic enhancements. Let $F\colon\T \to \T'$ be a
morphic triangle functor with enhancement $F_1\colon \T_1 \to \T'_1$.
Let $\X$ and $\X'$ be classes of sequences of $\T$ and $\T'$ satisfying
the assumptions of Section~\ref{ss:completion} and such that $F\X \subseteq \X'$.

\begin{lem} The functor $F$ induces a canonical triangle functor
$\widehat{F}\colon \widehat{\T_\X} \to \widehat{\T'_{\X'}}$.
\end{lem}

\begin{proof} This is a straightforward verification based on the fact that
$(F,F_1)$ is compatible with all the adjoints as checked in 
Lemma~\ref{lemma:morphic}.
\end{proof}

\subsection{Completion inside a compactly generated category}
\label{ss:morphic-compact-completion}

Let $\T$ be a compactly generated triangulated category with a morphic
enhancement $Q_0\colon \T \to \T_1$. By Lemma~\ref{lemma:compact}, we
have an induced morphic enhancement $Q_0\colon\T^c \to \T_1^c$ between
the subcategories of compact objects and
the inclusion $\T^c \to \T$ is morphic. Let $\X$ be a class of sequences
of $\T^c$ satisfying the hypotheses of Section~\ref{ss:completion}.
For a sequence $X$ in $\X$, we define $\hocolim_p X_p$ as
in Section~\ref{s:seq-completion}. As we have seen there, the facts
that each sequence $X\in\X$ is formed by compact objects and
that $\X$ is phantomless and stable under $\Sigma$
imply that for $X$, $Y$ in $\X$, we have a canonical bijection
\[
\widehat{\T^c_\X}(LX,LY) =\lim_p \colim_q \T^c(X_p,Y_q) \iso 
\T(\hocolim X, \hocolim Y).
\]
Thus we have a fully faithful functor
\[
F\colon \widehat{\T^c_\X} \to \T
\]
taking $LX$ to $\hocolim X$.
Clearly, $F$ is endowed with a canonical isomorphism $F\Sigma \iso \Sigma F$.

\begin{lem}\label{le:morphic-compact-completion}
  $F$ is a triangle functor.
\end{lem}

\begin{proof} Let 
\begin{equation} \label{eq:standard-triangle}
\xymatrix{
LP_1 X \ar[r] & LP_0 X \ar[r] & LP_{-1} X \ar[r] & \Sigma L P_0 X}
\end{equation}
be the standard triangle associated with a sequence $X$ of $\T_1^c$ such
that $P_0 X$ and $P_1 X$ belong to $\X$. Put
\[
Y=\hocolim_p X_p
\]
in $\T_1$. Using the fact that $P_1$, $P_0$ and $P_{-1}$ commute with
coproducts, it is easy to see that the standard triangle
\[
\xymatrix{
P_1 Y \ar[r] & P_0 Y \ar[r] & P_{-1} Y \ar[r] & \Sigma P_1 Y}
\]
is isomorphic to the image of the triangle (\ref{eq:standard-triangle})
under $F$.
\end{proof}

\subsection{Completions of morphic enhancements}
\label{ss:completion-morphic-enrich}

Let $\T$ be a triangulated category with a morphic enhancement
$Q_0\colon \T\to\T_1$ and let $\X$ be a class of sequences of $\T$
satisfying the assumptions of Section~\ref{ss:completion}. It is
natural to ask whether the triangulated category $\widehat{\T_\X}$
admits a morphic enhancement given by a completion of $\T_1$. Clearly,
the class of sequences $\X_1$ of $\T_1$ needed for this is formed by
the sequences $X$ such that $P_1 X$ and $P_0 X$ belong to $\X$. Let
$Q^j_0\colon \T_1 \to \T_2$, $j=1,2$, be morphic enhancements of $\T_1$ 
such that the categories $\T$, $\T_1$, $\T_2$ together with the 
given functors and their needed adjoints
satisfy the axioms for the first three floors of an epivalent tower of
triangulated categories \cite{Ke1991}. Then it is not hard to show
that $\X_1$ satisfies assumptions a), b), and c) of
Section~\ref{ss:completion}.  We cannot expect that $\X_1$ is
phantomless in general but this is the case in many examples. Indeed,
if $\T$ is the perfect derived category of a right coherent ring
$\Lambda$, then the canonical morphic enhancement for $\T$ is the
perfect derived category $\T_1$ of the ring $\Lambda_1$ of upper
triangular $2\times 2$ matrices over $\Lambda$. Notice that
$\Lambda_1$ is still right coherent. If $\X$ is the class of bounded
Cauchy sequences in $\T$, then $\X_1$ is easily seen to be the class
of bounded Cauchy sequences of $\T_1$. So in this example, $\X_1$ is
still phantomless. We can iterate this process to see that the
epivalent tower associated with the bounded derived category of
$\mod\Lambda$ is the bounded Cauchy completion of the tower associated
with the perfect derived category of $\Lambda$. An analogous statement
holds for the stable derivators, defined on the $2$-category of finite
directed categories, associated with the bounded derived category of
$\mod\Lambda$ and with the perfect derived category of $\Lambda$,
cf.\ the appendix \cite{Ke2007} to \cite{Maltsiniotis2007} for these
derivators.

Let $\M\X$ be the class of sequences $(X_{p1} \to X_{p0})$ of morphisms
of $\T$ such that $(X_{p1})$ and $(X_{p0})$ belong to $\X$.

\begin{lem} $\X_1$ is phantomless in $\T_1$ if and only if $\M\X$ is
phantomless in $\M\T$.
\end{lem}

\begin{proof} Let $X$ and $Y$ be in $\X_1$. By part c) of Lemma~\ref{prop:properties},
for all $p,q\geq 0$, we have an exact sequence
\[
\T(\Sigma P_1 X_p, P_0 Y_q) \to \T_1(X_p, Y_q) \to \M\T(MX_p, MY_q) \to 0.
\]
We pass to the colimit over $q$ to get an inverse system of exact sequences
\[
A_p \to B_p \to C_p \to 0.
\]
Since $\lim^1$ is right exact, it induces an exact sequence
\[
{\lim}^1 A_p \to {\lim}^1 B_p \to {\lim}^1 C_p \to 0.
\]
Now since $\X$ is phantomless, the group $\lim^1 A_p$ vanishes.
\end{proof}

\end{document}